\documentclass{article}
\ifdefined\pdfoutput\pdfoutput=1\fi
\usepackage{arxiv}
\usepackage{iftex}
\ifPDFTeX
  \usepackage[utf8]{inputenc}
  \usepackage[T1]{fontenc}
\else
  \usepackage{fontspec}
\fi
\usepackage{amsmath,amssymb,amsthm}
\usepackage{stmaryrd}
\usepackage{mathtools}
\usepackage{enumitem}
\usepackage{booktabs}
\usepackage{hyperref}
\hypersetup{
	hidelinks,
	linktoc = all,
	pdfdisplaydoctitle,
	breaklinks,
	pdfstartview = Fit,
	unicode,
	pdftitle={A vector logic for intensional formal semantics},
	pdfsubject={vector logic, formal semantics, intensional semantics},
	pdfauthor={Daniel Quigley},
	pdfcreator = {LuaLaTeX},
	pdfkeywords={formal semantics, intensional model, semantic space, vector logic, measure frames, graded modality, mathematical linguistics},
}
\usepackage[activate={true,nocompatibility},final]{microtype}
\usepackage{doi}
\usepackage{array}
\usepackage{cleveref}
\usepackage[most]{tcolorbox}
\usepackage{tikz}
\usetikzlibrary{cd,arrows.meta,positioning,calc}
\DeclareMathOperator{\Hom}{Hom}
\DeclareMathOperator{\IM}{im}
\DeclareMathOperator{\spn}{span}
\newcommand{\R}{\mathbb{R}}
\newcommand{\N}{\mathbb{N}}
\newcommand{\Free}[1]{\mathcal{F}_{#1}}
\newcommand{\Sp}[1]{\mathcal{S}_{#1}}
\newcommand{\Fr}{\mathbf{F}}
\newcommand{\ind}{\mathbf{1}}
\newcommand{\bA}{\mathbf{A}}
\newcommand{\bv}{\mathbf{v}}
\newcommand{\bb}{\mathbf{b}}
\newcommand{\bP}{\mathbf{P}}
\newcommand{\bB}{\mathbf{B}}
\newcommand{\Mint}{\mathcal{M}_{int}}
\newcommand{\Mext}{\mathcal{M}_{ext}}
\newcommand{\Ms}{\mathcal{M}_{\mathcal{S}}}
\newcommand{\den}[1]{\llbracket #1 \rrbracket}

\definecolor{boxcolor1}{RGB}{255,254,230}
\definecolor{boxcolor2}{RGB}{252,227,215}
\definecolor{boxcolor3}{RGB}{225,240,227}
\definecolor{boxcolor4}{RGB}{227,235,246}

\tcbset{
	envbox/.style={
		breakable,
		before skip=1.5\topskip,
		after skip=1.5\topskip,
		left skip=0pt,
		right skip=0pt,
		left=4pt,
		right=4pt,
		top=2pt,
		bottom=2pt,
		lefttitle=4pt,
		righttitle=4pt,
		toptitle=2pt,
		bottomtitle=2pt,
		sharp corners,
		boxrule=0pt,
		titlerule=.4pt,
		coltitle=black,
		colframe=darkgray,
		coltext=black,
		fonttitle=\bfseries,
	}
}
\newtcolorbox{definitionbox}[1][]{envbox, colback=boxcolor1, colbacktitle=boxcolor1, #1}
\newtcolorbox{theorembox}[1][]{envbox, colback=boxcolor2, colbacktitle=boxcolor2, #1}
\newtcolorbox{lemmabox}[1][]{envbox, colback=boxcolor3, colbacktitle=boxcolor3, #1}
\newtcolorbox{propositionbox}[1][]{envbox, colback=boxcolor3, colbacktitle=boxcolor3, #1}
\newtcolorbox{corollarybox}[1][]{envbox, colback=boxcolor4, colbacktitle=boxcolor4, #1}

\newtheoremstyle{mystyle}{3pt}{3pt}{\itshape}{}{\bfseries}{.}{.5em}{}
\theoremstyle{mystyle}
\newtheorem{example}{Example}[section]
\newtheorem{remark}{Remark}[section]

\newtheorem{definition}{Definition}[section]
\renewenvironment{definition}[1][]{%
	\refstepcounter{definition}%
	\begin{definitionbox}[title=Definition~\thedefinition\ifx\relax#1\relax\else\ (#1)\fi]%
	}{%
	\end{definitionbox}
}
\newtheorem{theorem}{Theorem}[section]
\renewenvironment{theorem}[1][]{%
	\refstepcounter{theorem}%
	\begin{theorembox}[title=Theorem~\thetheorem\ifx\relax#1\relax\else\ (#1)\fi]%
	}{%
	\end{theorembox}
}
\newtheorem{lemma}{Lemma}[section]
\renewenvironment{lemma}[1][]{%
	\refstepcounter{lemma}%
	\begin{lemmabox}[title=Lemma~\thelemma\ifx\relax#1\relax\else\ (#1)\fi]%
	}{%
	\end{lemmabox}
}
\newtheorem{proposition}{Proposition}[section]
\renewenvironment{proposition}[1][]{%
	\refstepcounter{proposition}%
	\begin{propositionbox}[title=Proposition~\theproposition\ifx\relax#1\relax\else\ (#1)\fi]%
	}{%
	\end{propositionbox}
}
\newtheorem{corollary}{Corollary}[section]
\renewenvironment{corollary}[1][]{%
	\refstepcounter{corollary}%
	\begin{corollarybox}[title=Corollary~\thecorollary\ifx\relax#1\relax\else\ (#1)\fi]%
	}{%
	\end{corollarybox}
}

\title{A vector logic for intensional formal semantics}

\author{ \href{https://orcid.org/0009-0004-7957-1806}{\includegraphics[scale=0.06]{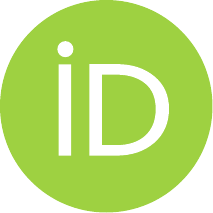}\hspace{1mm}Daniel Quigley} \\
	Center for Possible Minds\\
	Indiana University Bloomington\\
	Bloomington, IN 47408 \\
	\texttt{dgquigle@iu.edu} \\
}

\renewcommand{\headeright}{Preprint}
\renewcommand{\undertitle}{Preprint}
\renewcommand{\shorttitle}{A vector logic for intensional formal semantics}

\begin{document}
\maketitle

\begin{abstract}
    Formal semantics and distributional semantics are distinct approaches to linguistic meaning: the former models meaning as reference via model-theoretic structures; the latter as vectors in high-dimensional spaces shaped by usage. This paper establishes which part of intensional formal semantics admits a linear vector-space encoding. Kripke-style intensional models, with any finite collection of index sorts collected in a compound index space, embed injectively into vector spaces: primitive domains go to free carriers; intensions and other functions go to linear operators. Semantic functions lift to unique multilinear maps on the free carriers, and composition is preserved. The operator encoding of a function domain compresses its free carrier, and we characterize the functionals: a Boolean-valued functional of a power set acts linearly on operator encodings exactly when it is constant, an ultrafilter indicator, or the complement of one, so that on finite domains the nonconstant ones are Montague's individuals and their negations. Determiners over a restrictor of two or more elements, modal operators over two or more accessible indices, and attitude operators over two or more alternatives are outside the linear regime, and take, instead, the form of a linear accumulation followed by a decision. Modality is defined uniformly over measure frames, in which counting measure recovers Kripke semantics at every cardinality, and continuous measures make necessity truth almost everywhere; we give the correspondence conditions for the axioms D, T, B, 4, and 5 under measures, the Kronecker factorization of accessibility over compound indices, and the reading of measure-based modality as graded modality.
\end{abstract}


\section{Introduction}\label{sec:introduction}

Formal semantics and distributional semantics are two mathematically distinct frameworks for representing linguistic meaning. The formal approach, rooted in Frege, Tarski, and Montague \cite{Montague1974,Tarski1931}, models meaning as reference: linguistic expressions map to truth values via model-theoretic structures \cite{Mendelson1997,Chang1977}, and the denotation of an expression is its correspondence with entities, relations, or truth values in a model. The distributional approach, following Harris's hypothesis \cite{harris1979mathematical} and Firth's dictum that ``you shall know a word by the company it keeps'' \cite{firth2957}, represents meaning through vectors in high-dimensional spaces, where semantic relationships emerge from contextual co-occurrence patterns \cite{Lenci2022,Boleda2020} and conceptual similarity \cite{Gardenfors2014,Marcolli2023,Gardenfors2000}. Linguistics and cognitive science use both modes of description, logical inference over propositional content and relational reasoning over graded similarity, which raises the question of which parts of a formal semantics admit a vector-space encoding, and in what form.

This paper turns to intensional semantics, the extension of formal semantics that handles modality, tense, and propositional attitudes, by interpreting expressions relative to indices such as possible worlds, times, and locations. Building on the extensional construction of \cite{Quigley2025}, we embed Kripke-style intensional models into vector spaces. A free carrier and an operator carrier attach to each type. The free carrier has a basis indexed by the domain, and on it every semantic function lifts to a unique multilinear map, with composition preserved. The operator carrier represents a function as a linear map from the free carrier of its argument type; intensions become linear operators from an index space to an extensional carrier, and function application becomes a bilinear evaluation map. The two carriers agree on primitive types and differ on function types, where the operator carrier compresses the free one. The central result characterizes the functionals: a Boolean-valued functional of a power set acts linearly on operator encodings if and only if it is constant, an ultrafilter indicator, or the complement of one, so that on finite domains the nonconstant functionals of this kind at the type of generalized quantifiers are Montague's individuals and their negations. Every determiner with a restrictor of two or more elements, every modal operator with two or more accessible indices, and every attitude operator with two or more alternatives falls outside the linear regime, an observation of \cite{Grefenstette2013}, that quantifiers are not multilinear.



\section{Prior work}\label{sec:priorwork}

Prior work in compositional distributional semantics has progressed in various ways, each with characteristic limitations. The DisCoCat framework \cite{Coecke2010} maps pregroup grammar functorially to finite-dimensional vector spaces via compact-closed categories, with subsequent extensions accommodating the Lambek calculus \cite{Lambek1958,Lambek1961} to handle structural ambiguity \cite{Coecke2013}. Compact closure holds in finite dimension and fails in infinite dimension, where the unit of the adjunction is a divergent sum; and a pregroup is a preorder, with at most one reduction per pair of types, which is what collapses distinct parses into one morphism. Domain-theoretic approaches \cite{Burnistov2023} establish correspondences through continuous approximation, while vector lattice models encode Kripke frames as $K$-algebras \cite{Greco2020,Palmigiano2024}; neither addresses the interface with compositional semantics directly. Usage-based composition models \cite{MitchellLapata2008,Baroni2014} and neural approaches capture distributional similarity and priming effects without a guarantee that truth-value computation is preserved. \cite{Mao2022} demonstrate that distributional models with spreading activation outperform neural networks on compositional generalization tasks; \cite{Amigo2022} establish information-theoretic properties for compositional distributional semantics that improve textual similarity performance. These empirical successes leave open the question of structural correspondence with formal semantics.

The vector logic of Mizraji \cite{Mizraji1992,Mizraji1996,Mizraji2008}, elaborated by Westphal and Hardy \cite{Westphal2005}, embeds Boolean operations as matrices acting on vectorized truth values: connectives become matrix operations, truth values become basis vectors, and truth-functional composition becomes matrix-vector multiplication. The same encoding of truth values, with negation as the swap matrix, appears in \cite{Coecke2010}, and \cite{Grefenstette2013} extends it to a quantifier-free predicate calculus: domain elements as one-hot vectors, predicates as matrices whose columns are truth vectors, relations as tensors, and the connectives as matrices on tensor powers of the truth space. Tensor-product binding of symbolic structures goes back to \cite{Smolensky1990}. The extensional theorem of \cite{Quigley2025} states this encoding for arbitrary typed models and domains of any cardinality: typed domains embed injectively, and functions lift to (multi)linear maps preserving composition. \cite{Grefenstette2013} also proves that universal and existential quantification are not multilinear on this encoding. Section~\ref{sec:homomorphism} turns that observation into a characterization of the functionals that act linearly on encoded functions, and Section~\ref{sec:modalops} gives the form of those that do not. Within DisCoCat, quantifiers have been handled by adjoining bialgebra structure to the vector spaces \cite{HedgesSadrzadeh2019}; the accumulate-and-decide form below is the analogue for the present encoding.

The measure-based modality of Section~\ref{sec:modalops} is a neighborhood semantics \cite{Chellas1980,Pacuit2017}, whose neighborhoods are the sets of full measure within the accessible set; the relation between measure semantics and topological semantics is studied in \cite{Lando2012}. Its reading as graded modality follows the probability operators of \cite{Yalcin2010} and the degree semantics of \cite{Lassiter2017}. Compound indices with independent accessibility relations are product frames in the sense of \cite{gabbay2003many,KURUCZ2007869}. The present work adds to these an encoding in which the index space, the accessibility kernel, and the modal operators are objects of a linear algebra, and checks for linearity within that algebra.

\section{Formal and vector-space semantics}\label{sec:formalanddistr}

We briefly fix notation for both frameworks before establishing their connection.

A \textbf{formal system} $\mathcal{M}_{\mathcal{FS}} = \langle \mathcal{FL}, \models, \vdash_{\mathcal{C}} \rangle$ consists of a formal language $\mathcal{FL}$ of well-formed formulas, a model-theoretic entailment relation $\models$, and proof-theoretic derivability $\vdash_{\mathcal{C}}$ in a calculus $\mathcal{C}$. The application to natural language proceeds by translating a natural language fragment $\mathcal{NL}$ into $\mathcal{FL}$, then interpreting formulas relative to a model $\mathcal{M}$: for any well-formed formula $\varphi$, the denotation $\den{\varphi}^{\mathcal{M}}$ assigns a semantic value. We adopt the Heim and Kratzer framework \cite{vonFintel2023,Heim1998}, employing the simply typed lambda calculus \cite{Church1940} as the calculus $\mathcal{C}$.

The distributional approach represents meaning as vectors in high-dimensional space, grounded in the hypothesis that words with similar contexts have similar meanings \cite{firth2957,harris1979mathematical}. This approach has roots in linguistics, psychology, and computer science, with applications spanning cognitive science \cite{Gardenfors2000}, machine learning \cite{Bishop2006,Pavlick2022}, and computational linguistics \cite{Lenci2023,Lenci2022}. Word vectors are learned from corpora via co-occurrence statistics \cite{Bullinaria2012}, neural methods \cite{mikolov2013efficientestimationwordrepresentations,pennington2014glove}, or matrix factorization \cite{Arora2012}, such that semantic similarity corresponds to vector proximity, typically cosine similarity.

\begin{definition}[Semantic space]\label{def:semspace}
    A \textbf{semantic space} $\mathcal{S} = \langle \mathcal{V}, \langle \cdot | \cdot \rangle \rangle$ is a real inner product space.
\end{definition}

In prose, the vector semantic model is a structure:
$$\Ms = \left\langle \begin{array}{c} \text{meaning space} \end{array}, \begin{array}{c} \text{similarity metric} \end{array} \right\rangle.$$

The carriers of Section~\ref{sec:homomorphism} are free vector spaces on the semantic domains, equipped with an inner product for which the domain basis is orthonormal; where these carriers are finite-dimensional, they are Hilbert spaces, and every proof below uses their vector-space and inner-product structure alone.

The encoding established below operates on typed semantic domains, and on word vectors only through them. In formal semantics, the interpretation function $\mathcal{I}$ assigns to each lexical constant $c$ of type $\tau$ a denotation $\mathcal{I}(c) \in \mathcal{D}_\tau$: a proper name denotes an entity in $\mathcal{D}_e$; a sentence denotes a truth value in $\mathcal{D}_t$; a predicate denotes a function in $\mathcal{D}_{\langle e, t \rangle}$. Our construction embeds these typed domains into vector spaces: entities map to basis vectors; truth values map to orthonormal vectors in $\R^2$; functions map to linear operators. This is compatible with, and distinct from, an empirically learned word embedding. The free carrier of the entity domain is the input of every embedding layer, and whether a learned geometry of lower dimension carries a given lexicon exactly is a rank condition on that geometry; the present paper works with the carriers themselves.

\section{Vector logic for extensional formal semantics}\label{sec:extensional}

The extensional model of formal semantics is a typed structure $\Mext = \langle (\mathcal{D}_\tau)_{\tau \in \mathcal{T}}, \mathcal{I} \rangle$, where $(\mathcal{D}_\tau)_{\tau \in \mathcal{T}}$ is an indexed family of domains for each type $\tau$, and $\mathcal{I}$ is an interpretation function assigning semantic values to non-logical constants. The primitive types are $e$ (entities) and $t$ (truth values $\{0,1\}$). Complex types are constructed via the function type constructor: if $\alpha$ and $\beta$ are types, then $\langle \alpha, \beta \rangle$ is the type of functions from $\alpha$ to $\beta$. The vector space model $\Ms$ assigns to each domain a semantic space, and the following theorem states the extensional compatibility result of \cite{Quigley2025}, which the present work continues.

\begin{theorem}[Extensional lifting]\label{thm:exthomo}
    For every extensional model $\Mext$ there exist injective mappings $H = \{h_\tau\}_{\tau \in \mathcal{T}}$, $h_\tau : \mathcal{D}_\tau \to \mathcal{S}_{\mathcal{D}_\tau}$, into vector spaces such that every semantic function $f : \mathcal{D}_A \to \mathcal{D}_B$ has a unique linear lift $h_f : \mathcal{S}_{\mathcal{D}_A} \to \mathcal{S}_{\mathcal{D}_B}$ with $h_f(h_A(a)) = h_B(f(a))$ for all $a \in \mathcal{D}_A$; $n$-ary functions lift to multilinear maps, and the lift of a composite is the composite of the lifts.
\end{theorem}

The construction embeds every domain, function domains included, by basis vectors. The domain $\mathcal{D}_e = \{e_i \mid i \in I\}$ maps to a vector space $\mathcal{S}_{\mathcal{D}_e}$ with basis $\{\bb_i \mid i \in I\}$ via $h_e(e_i) = \bb_i$, and every other domain is treated in the same way, so that $\mathcal{S}_{\mathcal{D}_\tau}$ is the free vector space on $\mathcal{D}_\tau$. Injectivity follows from the distinctness of basis vectors; the uniqueness of the lift, its multilinearity, and the preservation of composition are the universal property of the free construction, since a linear map out of a free space is determined by its values on the basis. Relations and set-theoretic operations on $\mathcal{D}_e$ transfer along $h_e$, relative to the image. Because the free space on a nonempty set is strictly larger than the set, the injections are never surjective, and denotation is confined to the images. Section~\ref{sec:homomorphism} keeps this free carrier and adds a second carrier for function types.

The domain $\mathcal{D}_t = \{1,0\}$ maps to orthonormal basis vectors in $\R^2$, following the vectorization of truth values of \cite{Mizraji1992} and the truth space of \cite{Coecke2010}:
$$
h_t(1) = \bb_1 = \begin{bmatrix} 1 \\ 0 \end{bmatrix}, \qquad h_t(0) = \bb_0 = \begin{bmatrix} 0 \\ 1 \end{bmatrix}.
$$

Logical operators become matrices. For an $n$-ary operator $\mathrm{NOP} : \mathcal{D}_t^n \to \mathcal{D}_t$, there exists a $2 \times 2^n$ matrix $M$ such that
$$
\mathrm{NOP}'(h_t(t_1), \ldots, h_t(t_n)) = M \cdot \Bigl( \bigotimes_{i=1}^n h_t(t_i) \Bigr) = h_t(\mathrm{NOP}(t_1, \ldots, t_n)),
$$
which is again the universal property, applied to the basis $\{\bb_{t_1} \otimes \cdots \otimes \bb_{t_n}\}$ of $(\R^2)^{\otimes n}$. Negation corresponds to the matrix $N = \left[\begin{smallmatrix} 0 & 1 \\ 1 & 0 \end{smallmatrix}\right]$ and conjunction to $C = \left[\begin{smallmatrix} 1 & 0 & 0 & 0 \\ 0 & 1 & 1 & 1 \end{smallmatrix}\right]$ in the ordering $\bb_1 \otimes \bb_1, \bb_1 \otimes \bb_0, \bb_0 \otimes \bb_1, \bb_0 \otimes \bb_0$; these are the connective tensors of \cite{Grefenstette2013}. A monadic predicate $P \in \mathcal{D}_{\langle e, t\rangle}$ over a finite domain becomes the $2 \times |\mathcal{D}_e|$ matrix whose $i$-th column is $h_t(P(e_i))$, and $P$ applied to $e_i$ is that matrix applied to $\bb_i$.

The extension to intensional models requires representing Kripke frame structures, sets of indices equipped with accessibility relations, within this vector logic, together with the intensions that take indices as arguments; the present work accommodates these structures, and determines which operations on them remain linear.

\section{Intensional model of semantics}\label{sec:intensional}

Intensional semantics introduces parametrization: the intension of an expression is a function that takes a parameter (an index) as input and returns the extension of the expression at that index. This additional functional layer captures distinctions that extensional semantics cannot: contingently equivalent propositions receive distinct intensions; reference varies across worlds and times; counterfactual and modal reasoning become quantification over accessible indices; propositional attitudes become quantification over alternatives, up to the familiar limitation that necessarily equivalent propositions still coincide, whose refinement lies outside the scope of this paper.

\begin{definition}[Index sorts]\label{def:sorts}
    Let $\Sigma$ be a nonempty finite set of \textbf{index sorts}. Each sort $\sigma \in \Sigma$ represents a dimension of intensional variation.
\end{definition}

Standard choices of index sorts include possible worlds $w$ (for modal and counterfactual reasoning), times $\iota$ (for temporal semantics), and locations $\ell$ (for spatial semantics); further sorts enter as an application requires. The framework accommodates any finite collection of index sorts as required by the linguistic phenomena under analysis.

\begin{definition}[Intensional types]\label{def:types}
    The set of \textbf{intensional types} $\mathcal{T}$ is the smallest set such that:
    \begin{itemize}
    	\item $e \in \mathcal{T}$ is the type of entities;
    	\item $t \in \mathcal{T}$ is the type of truth values;
    	\item for each index sort $\sigma \in \Sigma$, $\sigma \in \mathcal{T}$ is a primitive index type;
    	\item $s \in \mathcal{T}$ is the type of compound indices;
    	\item if $\alpha, \beta \in \mathcal{T}$, then $\langle \alpha, \beta \rangle \in \mathcal{T}$.
    \end{itemize}
\end{definition}

We partition the types as follows:
\begin{itemize}
	\item $\mathcal{T}_{ext} = \{e, t\}$, the extensional primitive types;
	\item $\mathcal{T}_{idx} = \Sigma \cup \{s\}$, the index types;
	\item $\mathcal{T}_0 = \mathcal{T}_{ext} \cup \mathcal{T}_{idx}$, all primitive types;
	\item $\mathcal{T}_+ = \mathcal{T} \setminus \mathcal{T}_0$, the complex (function) types.
\end{itemize}

With $s$ a type, the system is a two-sorted type theory in the sense of Gallin \cite{Gallin2016}, in which the intensional types $\langle s, \tau \rangle$ are well formed and, for discrete sorts, the modal operators of Definition~\ref{def:intendenot} are definable by quantification over indices. We keep the modal clauses primitive, since on measured sorts they involve the measures of the frames, and since Section~\ref{sec:modalops} treats them as operators in their own right.

Each accessibility relation carries a measure, which will let one definition of modality cover discrete and continuous index domains alike.

\begin{definition}[Measure frame]\label{def:frame}
    For an index sort $\sigma \in \Sigma$, a \textbf{measure frame} $\mathcal{F}_\sigma = (\mathcal{D}_\sigma, \mathcal{M}_\sigma, \mathcal{R}_\sigma, \mu_\sigma)$ consists of a nonempty set $\mathcal{D}_\sigma$ of indices of sort $\sigma$; a $\sigma$-algebra $\mathcal{M}_\sigma$ of subsets of $\mathcal{D}_\sigma$ containing every singleton; an accessibility relation $\mathcal{R}_\sigma \subseteq \mathcal{D}_\sigma \times \mathcal{D}_\sigma$, measurable in the product $\sigma$-algebra, with accessible sets $\mathcal{R}_\sigma(w) = \{w' : \mathcal{R}_\sigma(w, w')\}$; and a family $\mu_\sigma = (\mu_{\sigma, w})_{w \in \mathcal{D}_\sigma}$ of measures on $(\mathcal{D}_\sigma, \mathcal{M}_\sigma)$. The sort is \textbf{discrete} when $\mathcal{M}_\sigma$ is the power set and each $\mu_{\sigma, w}$ is counting measure; otherwise, the sort is \textbf{measured}, and we require either that all $\mu_{\sigma, w}$ coincide with a single $\sigma$-finite measure $\mu_\sigma$, or that each $\mu_{\sigma, w}$ is a finite measure and $w \mapsto \mu_{\sigma, w}(B)$ is measurable for each $B \in \mathcal{M}_\sigma$.
\end{definition}

A discrete measure frame is a Kripke frame $(\mathcal{D}_\sigma, \mathcal{R}_\sigma)$; every Kripke frame arises in this way. In the measured case, the family $\mu_\sigma$ is a kernel, finite or constant in $w$; the examples below take a single $\sigma$-finite measure $\mu$ on $\mathcal{D}_\sigma$, Lebesgue measure on $\R$ for continuous time, and then only the restriction of $\mu$ to $\mathcal{R}_\sigma(w)$; Section~\ref{sec:graded} uses the finite kernels.

\begin{definition}[Intensional model]\label{def:intenmod}
    An \textbf{intensional model} $\Mint$ is a tuple:
    $$\Mint = \langle (\mathcal{D}_\tau)_{\tau \in \mathcal{T}}, (\mathcal{F}_\sigma)_{\sigma \in \Sigma}, S, \mathcal{I} \rangle$$
    
    where the components are specified as follows:
    
    \begin{itemize}
    	\item for each type $\tau \in \mathcal{T}$, the model determines a domain $\mathcal{D}_\tau$:
    	\begin{itemize}
    		\item $\mathcal{D}_e$ is a nonempty set of entities;
    		\item $\mathcal{D}_t = \{0, 1\}$ is the set of truth values;
    		\item for each index sort $\sigma \in \Sigma$, $\mathcal{D}_\sigma$ is the nonempty set of indices of the measure frame $\mathcal{F}_\sigma$;
    		\item $\mathcal{D}_s = S$, the set of compound indices below;
    		\item for function types: $\mathcal{D}_{\langle \alpha, \beta \rangle} = \mathcal{D}_\beta^{\mathcal{D}_\alpha}$, the set of all functions from $\mathcal{D}_\alpha$ to $\mathcal{D}_\beta$;
    	\end{itemize}
    
    	\item for each index sort $\sigma \in \Sigma$, $\mathcal{F}_\sigma = (\mathcal{D}_\sigma, \mathcal{M}_\sigma, \mathcal{R}_\sigma, \mu_\sigma)$ is a measure frame;
    
    	\item the set of compound indices is the Cartesian product
    	$$S = \prod_{\sigma \in \Sigma} \mathcal{D}_\sigma,$$
    	carrying the product $\sigma$-algebra $\bigotimes_{\sigma} \mathcal{M}_\sigma$. Each element $s \in S$ is a tuple $s = (s_\sigma)_{\sigma \in \Sigma}$ specifying a value for each index dimension; we write $s_\sigma$ for the $\sigma$-component of $s$, and $s[\sigma \mapsto w']$ for the tuple that agrees with $s$ except at $\sigma$, where it takes the value $w'$;
    
    	\item for each non-logical constant $c$ of type $\tau$ in the formal language $\mathcal{FL}$, the interpretation function $\mathcal{I}$ assigns an intension
    	$$\mathcal{I}(c) : S \to \mathcal{D}_\tau.$$
    	At each compound index $s \in S$, the extension of $c$ at $s$ is $\mathcal{I}(c)(s) \in \mathcal{D}_\tau$.
    \end{itemize}
\end{definition}

The technicalities of Definition~\ref{def:intenmod} may be characterized in prose as:\[\Mint = \left\langle \begin{array}{c} \text{type-indexed domains} \\ \text{of semantic values} \end{array}, \begin{array}{c} \text{sort-indexed frames} \\ \text{with accessibility and measure} \end{array}, \begin{array}{c} \text{compound} \\ \text{index space} \end{array}, \begin{array}{c} \text{intension} \\ \text{assignment} \end{array} \right\rangle\]

\begin{definition}[Intensional denotation]\label{def:intendenot}
    The \textbf{intensional denotation function} $\den{\cdot}^{\mathcal{M}, g, s}$ assigns to every expression $\varphi$ of $\mathcal{FL}$ a semantic value at compound index $s \in S$, defined recursively:
    
    \begin{enumerate}
    	\item if $\varphi$ is a constant: $\den{\varphi}^{\mathcal{M}, g, s} = \mathcal{I}(\varphi)(s)$;
    
    	\item let $V$ be the set of individual variables. The \textbf{assignment function} $g : V \to \mathcal{D}_e$ maps each variable to an entity. For any variable $x \in V$ and entity $k \in \mathcal{D}_e$, the \textbf{variant assignment} $g[x \mapsto k]$ is defined by
    	$$g[x \mapsto k](y) = \begin{cases} k & \text{if } y = x \\ g(y) & \text{if } y \neq x \end{cases}$$
    	If $\varphi$ is a variable $x \in V$: $\den{\varphi}^{\mathcal{M}, g, s} = g(x)$;
    
    	\item for function application, if $\varphi = \alpha(\beta)$ where $\alpha$ is of type $\langle \tau_1, \tau_2 \rangle$ and $\beta$ is of type $\tau_1$:
    	$$\den{\alpha(\beta)}^{\mathcal{M}, g, s} = \den{\alpha}^{\mathcal{M}, g, s}(\den{\beta}^{\mathcal{M}, g, s});$$
    
    	\item for lambda abstraction, if $\varphi = \lambda x_\tau . \psi$ where $x$ is a variable of type $\tau$ and $\psi$ is of type $\tau'$:
    	$$\den{\lambda x_\tau . \psi}^{\mathcal{M}, g, s} = \text{the function } f : \mathcal{D}_\tau \to \mathcal{D}_{\tau'} \text{ such that } f(d) = \den{\psi}^{\mathcal{M}, g[x \mapsto d], s};$$
    
    	\item for the necessity operator indexed by sort $\sigma \in \Sigma$, with $w = s_\sigma$:
    	$$\den{\Box_\sigma \varphi}^{\mathcal{M}, g, s} = 1 \iff \mu_{\sigma, w}\bigl(\{ w' \in \mathcal{R}_\sigma(w) : \den{\varphi}^{\mathcal{M}, g, s[\sigma \mapsto w']} = 0 \}\bigr) = 0;$$
    
    	\item the possibility operator is the dual $\Diamond_\sigma \varphi := \neg \Box_\sigma \neg \varphi$, which unfolds to
    	$$\den{\Diamond_\sigma \varphi}^{\mathcal{M}, g, s} = 1 \iff \mu_{\sigma, w}\bigl(\{ w' \in \mathcal{R}_\sigma(w) : \den{\varphi}^{\mathcal{M}, g, s[\sigma \mapsto w']} = 1 \}\bigr) > 0.$$
    \end{enumerate}
\end{definition}

For a discrete sort, clause 5 says that $\varphi$ holds at every accessible index and clause 6 that it holds at some accessible index, the classical clauses; for a measured sort, they say that $\varphi$ holds $\mu_{\sigma, w}$-almost everywhere on the accessible set, and on a subset of it of positive measure. We assume throughout that the truth set $\{s : \den{\varphi}^{\mathcal{M}, g, s} = 1\}$ of every formula is measurable in the product $\sigma$-algebra; for discrete sorts this is automatic, and for measured sorts it is a condition on the interpretation of the constants, preserved by the Boolean connectives and by the modal clauses, since the measure of the $\sigma$-section of a product-measurable set is a measurable function of the index.

\begin{remark}[Index-specific modality]\label{rem:indexmod}
    Clause 5 defines $\Box_\sigma$ as quantifying over indices that differ only in the $\sigma$-dimension. Thus:
    
    \begin{itemize}
        \item $\Box_w$ expresses metaphysical or alethic necessity (quantifying over accessible worlds, holding time and location fixed);
        \item $\Box_\iota$ expresses temporal necessity or ``always'' (quantifying over accessible times, holding world and location fixed);
        \item $\Box_\ell$ quantifies over accessible locations. 
    \end{itemize}

    This permits interaction between modal operators: $\Box_w \Box_\iota \varphi$ expresses that $\varphi$ holds at all accessible worlds and all accessible times, and Section~\ref{subsec:products} shows that the order of the two operators is immaterial when the sorts are independent.
\end{remark}

Under this framework (following \cite{vonFintel2023,Chierchia2000,Coppock2024}), intensional objects receive the following type assignments:
\begin{itemize}
	\item propositions are of type $\langle s, t \rangle$, functions from compound indices to truth values;
	\item individual concepts are of type $\langle s, e \rangle$, functions from compound indices to entities;
	\item properties are of type $\langle s, \langle e, t \rangle \rangle$, functions from compound indices to characteristic functions.
\end{itemize}
The intension-extension distinction is captured as follows: for a constant $c$ of extensional type $\tau$, the intension is $\mathcal{I}(c) \in \mathcal{D}_{\langle s, \tau \rangle}$, and the extension at index $s$ is $\mathcal{I}(c)(s) \in \mathcal{D}_\tau$.

\section{Homomorphism}\label{sec:homomorphism}

We now construct the vector space model $\Ms$ of an intensional model $\Mint$. Each type receives a free carrier and an operator carrier, which coincide on primitive types; every semantic function lifts to a multilinear map on the free carrier, and the operator carrier compresses it on function types.

\subsection{Carriers}\label{subsec:types}

For a type $\tau$, the carriers and the encoding are defined by recursion on the structure of types.

\begin{definition}[Carriers and encoding]\label{def:carriers}
    For each type $\tau \in \mathcal{T}$, the \textbf{free carrier} $\Free{\tau}$ is the real vector space with basis $\{\bb_a \mid a \in \mathcal{D}_\tau\}$, the space of finitely supported functions $\mathcal{D}_\tau \to \R$, equipped with the inner product for which this basis is orthonormal. The \textbf{operator carrier} $\Sp{\tau}$ and the \textbf{encoding} $h_\tau : \mathcal{D}_\tau \to \Sp{\tau}$ are defined by recursion on $\tau$:
    \begin{itemize}
    	\item for a primitive type $\tau \in \mathcal{T}_0$: $\Sp{\tau} = \Free{\tau}$ and $h_\tau(a) = \bb_a$;
    	\item for a function type $\tau = \langle \alpha, \beta \rangle$: $\Sp{\tau} = \Hom(\Free{\alpha}, \Sp{\beta})$, the vector space of all linear maps from $\Free{\alpha}$ to $\Sp{\beta}$, and $h_\tau(f)$ is the linear map determined on the basis of $\Free{\alpha}$ by
    	$$h_\tau(f)(\bb_a) = h_\beta(f(a)) \quad \text{for each } a \in \mathcal{D}_\alpha.$$
    \end{itemize}
    The \textbf{comparison map} $\varepsilon_\tau : \Free{\tau} \to \Sp{\tau}$ is the linear map with $\varepsilon_\tau(\bb_a) = h_\tau(a)$ for $a \in \mathcal{D}_\tau$; it is the identity for primitive $\tau$.
\end{definition}

For the truth type, $\Free{t} = \R^2$ with $\bb_1$ and $\bb_0$ as we saw in Section~\ref{sec:extensional}. For an index sort, $\Free{\sigma}$ has orthonormal basis $\{\bb_w \mid w \in \mathcal{D}_\sigma\}$, and for the compound index type $\Free{s}$ has basis $\{\bb_s \mid s \in S\}$; since $\Sigma$ is finite, $\Free{s} \cong \bigotimes_{\sigma \in \Sigma} \Free{\sigma}$ via $\bb_s = \bigotimes_{\sigma} \bb_{s_\sigma}$, which makes the compositional structure of compound indices explicit, so accessibility acts on one index dimension at a time.

The free carrier is in argument position throughout the recursion, because a linear map is determined by its values on a basis, and $\Sp{\alpha}$ carries a basis indexed by $\mathcal{D}_\alpha$ only when $\alpha$ is primitive. For primitive $\alpha$, we have $\Free{\alpha} = \Sp{\alpha}$ and, hence, $\Sp{\langle \alpha, \beta \rangle} = \Hom(\Sp{\alpha}, \Sp{\beta})$, which covers the first-order types of predicates, relations, propositions, individual concepts, and properties. The free carrier of every type is the carrier of the extensional construction (of Theorem~\ref{thm:exthomo}); the operator carrier, however, is the subject matter here.

\begin{lemma}[Injectivity]\label{lem:injfunc}
    For every type $\tau$, the encoding $h_\tau : \mathcal{D}_\tau \to \Sp{\tau}$ is injective.
\end{lemma}

\begin{proof}\label{prf:injfunc}
    By induction on $\tau$. For primitive $\tau$, distinct elements go to distinct basis vectors. For $\tau = \langle \alpha, \beta \rangle$, suppose $f \neq f'$ in $\mathcal{D}_\tau$. Then there exists $a \in \mathcal{D}_\alpha$ with $f(a) \neq f'(a)$, and by the injectivity of $h_\beta$,
    $$h_\tau(f)(\bb_a) = h_\beta(f(a)) \neq h_\beta(f'(a)) = h_\tau(f')(\bb_a),$$
    so $h_\tau(f) \neq h_\tau(f')$ as linear maps.
\end{proof}

Intensions are functions from the compound index space to extensional denotations. For an extensional type $\alpha$, the corresponding intensional type is $\langle s, \alpha \rangle$, with domain $\mathcal{D}_{\langle s, \alpha \rangle} = \{ g \mid g : S \to \mathcal{D}_\alpha \}$ and operator carrier $\Sp{\langle s, \alpha \rangle} = \Hom(\Free{s}, \Sp{\alpha})$. For each intension $g : S \to \mathcal{D}_\alpha$, the linear map $h_{\langle s, \alpha \rangle}(g) : \Free{s} \to \Sp{\alpha}$ acts on basis vectors by $h_{\langle s, \alpha \rangle}(g)(\bb_s) = h_\alpha(g(s))$.

\begin{example}[Propositions]\label{ex:prop}
    A proposition $p : S \to \mathcal{D}_t$ maps to a linear operator $h_{\langle s, t \rangle}(p) : \Free{s} \to \R^2$ defined by
    $$h_{\langle s, t \rangle}(p)(\bb_s) = \begin{cases} \bb_1 & \text{if } p(s) = 1 \\ \bb_0 & \text{if } p(s) = 0 \end{cases}$$
    The proposition is, thus, encoded as a linear map that classifies each index according to its truth value there. Writing $\bP_p$ for this operator, the row of $\bP_p$ read off by $\bb_1$ is the truth profile $\bv(p) \in \{0,1\}^S$ of $p$, with $\bv(p)_s = p(s)$, and the row read off by $\bb_0$ is its complement; the column sums of $\bP_p$ are all equal to $1$.
\end{example}

\begin{example}[Individual concepts]\label{ex:ic}
    An individual concept $c : S \to \mathcal{D}_e$ (for example, the classic ``the President of the United States'') maps to $h_{\langle s, e \rangle}(c) : \Free{s} \to \Free{e}$, sending
    $$h_{\langle s, e \rangle}(c)(\bb_s) = \bb_{c(s)}.$$
    The concept is encoded as a linear map that returns the vector representation of the individual who satisfies the concept at each index.
\end{example}

\begin{example}[Properties]\label{ex:propty}
    A property $P : S \to (\mathcal{D}_e \to \mathcal{D}_t)$ maps to a linear operator
    $$h_{\langle s, \langle e, t \rangle \rangle}(P) : \Free{s} \to \Hom(\Free{e}, \R^2).$$
    For each index $s$, the extension $P(s)$ is a characteristic function on entities, and $h_{\langle s, \langle e, t \rangle \rangle}(P)(\bb_s)$ is the linear operator encoding that characteristic function, the predicate matrix of Section~\ref{sec:extensional} when $\mathcal{D}_e$ is finite.
\end{example}

\begin{remark}[Varying]
    Some intensional frameworks allow the entity domain to vary across indices: for each $s \in S$, there is a (possibly proper) subset $\mathcal{D}_e(s) \subseteq \mathcal{D}_e$ of entities that exist at $s$. To accommodate this within our framework, we take $\mathcal{D}_e = \bigcup_{s \in S} \mathcal{D}_e(s)$ as the global entity domain, embed all entities into $\Free{e}$ via $h_e$ as before, and encode existence as a property of type $\langle s, \langle e, t \rangle \rangle$. The injectivity of $h_e$ on the global domain ensures that distinct entities receive distinct vector representations, even if they exist at different indices.
\end{remark}

\begin{remark}[Matrix form]\label{rem:tensor}
    For finite domains, there is a canonical isomorphism $\Hom(\Free{\alpha}, \Sp{\beta}) \cong \Free{\alpha}^* \otimes \Sp{\beta}$, where $\Free{\alpha}^*$ is the dual space with dual basis $\bb_a^*(\bb_{a'}) = \delta_{a, a'}$, and under it
    $$h_{\langle \alpha, \beta \rangle}(f) \longleftrightarrow \sum_{a \in \mathcal{D}_\alpha} \bb_a^* \otimes h_\beta(f(a)).$$
    In matrix terms, whatever the cardinality, the column of $h_{\langle \alpha, \beta \rangle}(f)$ at $a$ is $h_\beta(f(a))$; a predicate $P \in \mathcal{D}_{\langle e, t \rangle}$ is the matrix whose columns are $h_t(P(x))$ for $x \in \mathcal{D}_e$, as in the vector logic of \cite{Mizraji1992,Westphal2005} and the predicate tensors of \cite{Grefenstette2013}. For infinite domains the identification of $\Hom$ with a tensor product requires topological refinement \cite{Conway2007}; the $\Hom$ formulation needs none, which is why we adopt it.
\end{remark}

The comparison map records how the operator carrier compresses the free one.

\begin{lemma}[Comparison map]\label{lem:comparison}
    Let $\tau = \langle \alpha, \beta \rangle$.
    \begin{enumerate}
    	\item $\varepsilon_\tau$ maps $\Free{\tau}$ onto $\spn \IM(h_\tau) \subseteq \Sp{\tau}$.
    	\item If $\beta$ is primitive, then $\sum_i c_i\, h_\tau(f_i) = 0$ for a finite family of distinct $f_i \in \mathcal{D}_\tau$ if and only if, for every $a \in \mathcal{D}_\alpha$ and every $y \in \mathcal{D}_\beta$, $\sum_{i : f_i(a) = y} c_i = 0$.
    	\item If $\beta$ is primitive, then $\varepsilon_\tau$ is injective if and only if $|\mathcal{D}_\alpha| \le 1$ or $|\mathcal{D}_\beta| \le 1$; for finite domains,
    	$$\dim \ker \varepsilon_\tau = |\mathcal{D}_\beta|^{|\mathcal{D}_\alpha|} - |\mathcal{D}_\alpha|\,(|\mathcal{D}_\beta| - 1) - 1.$$
    \end{enumerate}
\end{lemma}

\begin{proof}
(1) is the definition of $\varepsilon_\tau$. 

\noindent For (2), the column at $a$ of $\sum_i c_i\, h_\tau(f_i)$ is $\sum_i c_i \bb_{f_i(a)} = \sum_{y} \bigl( \sum_{i : f_i(a) = y} c_i \bigr) \bb_y$, which vanishes if and only if every coefficient does. 

\noindent For (3), if $|\mathcal{D}_\alpha| \ge 2$ and $|\mathcal{D}_\beta| \ge 2$, pick $a \neq a'$ in $\mathcal{D}_\alpha$, $y \neq y'$ in $\mathcal{D}_\beta$, and any $f$; the four functions $f_{uv} := f[a \mapsto u, a' \mapsto v]$ with $u, v \in \{y, y'\}$ satisfy $h_\tau(f_{yy}) - h_\tau(f_{yy'}) - h_\tau(f_{y'y}) + h_\tau(f_{y'y'}) = 0$ by (2), so $\varepsilon_\tau$ has a nonzero kernel. If $|\mathcal{D}_\alpha| = 1$, then $h_\tau(f) = \bb_{f(a)}$ up to the identification $\Hom(\Free{\alpha}, \Free{\beta}) \cong \Free{\beta}$, and distinct $f$ give distinct basis vectors; if $|\mathcal{D}_\beta| = 1$ there is a single function. For the dimension, the marginal map $c \mapsto \bigl( \sum_{f(a) = y} c_f \bigr)_{a, y}$ from $\Free{\tau}$ to $\R^{\mathcal{D}_\alpha \times \mathcal{D}_\beta}$ has kernel $\ker \varepsilon_\tau$ by (2), and its image is the subspace of arrays whose row sums $\sum_y m_{a,y}$ are independent of $a$: each $\bb_f$ has all row sums equal to $1$, and the arrays $E_{a,y} - E_{a,y'}$, images of $\bb_{f[a \mapsto y]} - \bb_{f[a \mapsto y']}$, together with the image of one $\bb_f$ span that subspace, whose dimension is $|\mathcal{D}_\alpha||\mathcal{D}_\beta| - |\mathcal{D}_\alpha| + 1$.
\end{proof}


\subsection{Free lifts}\label{subsec:maps}

Having constructed the carriers, we now establish that every semantic function of $\Mint$ lifts to a multilinear map on the free carriers, and that the recursion of Definition~\ref{def:intendenot} computes in the carriers.

\begin{theorem}[Free lifting]\label{thm:lift}
    Let $f : \mathcal{D}_{\tau_1} \times \cdots \times \mathcal{D}_{\tau_n} \to \mathcal{D}_\tau$ be any function between domains of $\Mint$.
    \begin{enumerate}
    	\item There is a unique multilinear map $\tilde{f} : \Free{\tau_1} \times \cdots \times \Free{\tau_n} \to \Sp{\tau}$ with
    	$$\tilde{f}(\bb_{a_1}, \ldots, \bb_{a_n}) = h_\tau(f(a_1, \ldots, a_n)) \quad \text{for all } a_i \in \mathcal{D}_{\tau_i};$$
    	for $n = 1$ it is $h_{\langle \tau_1, \tau \rangle}(f)$.
    	\item Let $\Fr(f) : \Free{\tau_1} \otimes \cdots \otimes \Free{\tau_n} \to \Free{\tau}$ be the linear map with $\bb_{a_1} \otimes \cdots \otimes \bb_{a_n} \mapsto \bb_{f(a_1, \ldots, a_n)}$. Then $\tilde{f} = \varepsilon_\tau \circ \Fr(f)$ under the identification of multilinear maps with linear maps on the tensor product, and for unary $f : \mathcal{D}_\alpha \to \mathcal{D}_\beta$ and $g : \mathcal{D}_\beta \to \mathcal{D}_\gamma$,
    	$$\Fr(g \circ f) = \Fr(g) \circ \Fr(f).$$
    	\item If $\beta$ is primitive, then also $\widetilde{g \circ f} = \tilde{g} \circ \tilde{f}$.
    \end{enumerate}
\end{theorem}

The square
\begin{center}
    \begin{tikzcd}[row sep=large, column sep=large]
    	\mathcal{D}_\alpha \arrow[r, "f"] \arrow[d, "a \mapsto \bb_a"'] & \mathcal{D}_\beta \arrow[d, "h_\beta"] \\
    	\Free{\alpha} \arrow[r, "\tilde f = \varepsilon_\beta \circ \Fr(f)"'] & \Sp{\beta}
    \end{tikzcd}
\end{center}
commutes for every unary $f$.

\begin{proof}
\noindent (1) The tuples of basis vectors $(\bb_{a_1}, \ldots, \bb_{a_n})$ form a basis of $\Free{\tau_1} \otimes \cdots \otimes \Free{\tau_n}$, and a multilinear map is the same as a linear map on the tensor product, which is determined by its values on a basis and may take arbitrary values there. For $n = 1$ the prescription is that of Definition~\ref{def:carriers}. 

\noindent (2) Both sides of $\tilde{f} = \varepsilon_\tau \circ \Fr(f)$ send $\bb_{a_1} \otimes \cdots \otimes \bb_{a_n}$ to $h_\tau(f(a_1, \ldots, a_n))$; both sides of $\Fr(g \circ f) = \Fr(g) \circ \Fr(f)$ send $\bb_a$ to $\bb_{g(f(a))}$; uniqueness gives equality. 

\noindent (3) If $\beta$ is primitive, then $\varepsilon_\beta$ is the identity, $\tilde{f} = \Fr(f)$ lands in $\Free{\beta} = \Sp{\beta}$, and $\tilde{g} \circ \tilde{f}$ agrees with $\widetilde{g \circ f}$ on every $\bb_a$.
\end{proof}

Theorem~\ref{thm:lift} is the intensional form of Theorem~\ref{thm:exthomo}: on the free carriers, lifting is a functor, and every chain $f_n \circ \cdots \circ f_1$ of semantic functions lifts to the chain of its lifts. The operator encoding of a function is its free lift followed by the comparison map of the value type. The two carriers, therefore, agree along chains that pass through primitive types, and diverge exactly where a chain passes through a function type, which is the situation of Section~\ref{subsec:descent}.

\begin{proposition}[Recursion in the carriers]\label{prop:recursion}
    \begin{enumerate}
    	\item Application. For $f \in \mathcal{D}_{\langle \alpha, \beta \rangle}$ and $a \in \mathcal{D}_\alpha$, $h_\beta(f(a)) = h_{\langle \alpha, \beta \rangle}(f)(\bb_a)$. The evaluation map $\mathrm{ev} : \Hom(\Free{\alpha}, \Sp{\beta}) \times \Free{\alpha} \to \Sp{\beta}$, $(L, u) \mapsto L u$, is bilinear, so clause 3 of Definition~\ref{def:intendenot} computes as
    	$$h_{\tau_2}\bigl(\den{\alpha(\beta)}^{\mathcal{M}, g, s}\bigr) = h_{\langle \tau_1, \tau_2 \rangle}\bigl(\den{\alpha}^{\mathcal{M}, g, s}\bigr)\bigl(\bb_{\den{\beta}^{\mathcal{M}, g, s}}\bigr).$$
    	\item Abstraction. $h_{\langle \tau, \tau' \rangle}\bigl(\den{\lambda x_\tau . \psi}^{\mathcal{M}, g, s}\bigr)$ is the linear map determined by $\bb_d \mapsto h_{\tau'}\bigl(\den{\psi}^{\mathcal{M}, g[x \mapsto d], s}\bigr)$.
    	\item The intension of an expression $\xi$ of type $\tau$, $s \mapsto \den{\xi}^{\mathcal{M}, g, s}$, is encoded as the linear map $\Free{s} \to \Sp{\tau}$ with $\bb_s \mapsto h_\tau(\den{\xi}^{\mathcal{M}, g, s})$, and the extension of $\xi$ at $s$ is recovered by evaluating this map at $\bb_s$.
    	\item The truth-functional connectives compute by the matrices of Section~\ref{sec:extensional} on tensor powers of $\Free{t}$.
    \end{enumerate}
    Consequently, the recursion of Definition~\ref{def:intendenot}, with each argument carried in its free encoding, computes in the carriers with function application as evaluation.
\end{proposition}

\begin{proof}
    (1) is Definition~\ref{def:carriers} together with the bilinearity of evaluation; (2) and (3) restate the definition of $h_{\langle \tau, \tau' \rangle}$ and $h_{\langle s, \tau \rangle}$ on the functions produced by clauses 4 and 1; (4) is Section~\ref{sec:extensional}.
\end{proof}

If $\tau_1$ is a function type, the argument $\den{\beta}$ clause 3 as the basis vector $\bb_{\den{\beta}}$ of $\Free{\tau_1}$, and the operator encoding $h_{\tau_1}(\den{\beta}) \in \Sp{\tau_1}$ of the same argument is a different vector, related to the first by $\varepsilon_{\tau_1}$. Whether a functor of type $\langle \tau_1, \tau_2 \rangle$ can be applied to operator encodings is the question of descent.

\subsection{Descent}\label{subsec:descent}

A functional of a function domain acts on free encodings by Theorem~\ref{thm:lift}; we ask when it also acts, linearly, on operator encodings.

\begin{definition}[Descent]\label{def:descent}
    Let $\alpha$ be a type and $F : \mathcal{D}_\alpha \to \mathcal{D}_t$. The functional $F$ \textbf{descends} if there is a linear map $\Lambda : \Sp{\alpha} \to \Free{t}$ with
    $$\Lambda(h_\alpha(f)) = h_t(F(f)) \quad \text{for every } f \in \mathcal{D}_\alpha;$$
    equivalently, $h_{\langle \alpha, t \rangle}(F) = \Lambda \circ \varepsilon_\alpha$ for some linear $\Lambda$, or $\ker \varepsilon_\alpha \subseteq \ker h_{\langle \alpha, t \rangle}(F)$.
\end{definition}

\begin{center}
\begin{tikzcd}[row sep=large, column sep=large]
	\Free{\alpha} \arrow[r, "\varepsilon_\alpha"] \arrow[dr, "h_{\langle \alpha, t \rangle}(F)"'] & \Sp{\alpha} \arrow[d, dashed, "\Lambda"] \\
	& \Free{t}
\end{tikzcd}
\end{center}

The three formulations agree because a linear map out of $\Free{\alpha}$ factors through $\varepsilon_\alpha$ if and only if it vanishes on $\ker \varepsilon_\alpha$, in which case it factors through $\spn \IM(h_\alpha)$, and a linear map defined on that subspace extends to $\Sp{\alpha}$. For primitive $\alpha$, the comparison map is the identity, and every functional descends. The interesting case is $\alpha = \langle \alpha_1, t \rangle$, where $\mathcal{D}_\alpha$ is the power set of $D := \mathcal{D}_{\alpha_1}$ via characteristic functions; we treat at once any Boolean algebra $\mathcal{B}$ of subsets of $D$, so as to cover the measurable propositions of a measured sort. For $X \in \mathcal{B}$, write $\bP_X := h_\alpha(X) \in \Hom(\Free{D}, \Free{t})$, the linear map with $\bb_a \mapsto h_t([a \in X])$, and $\bv(X) \in \{0,1\}^D$ for the indicator of $X$; as in Example~\ref{ex:prop}, the row of $\bP_X$ read off by $\bb_1$ is $\bv(X)$, and every column of $\bP_X$ sums to $1$.

\begin{lemma}[Encoded sets]\label{lem:encodedsets}
    Let $\mathcal{B}$ be a Boolean algebra of subsets of $D$, with $D \in \mathcal{B}$.
    \begin{enumerate}
    	\item For a finite family of distinct $X_i \in \mathcal{B}$, $\sum_i c_i \bP_{X_i} = 0$ if and only if $\sum_i c_i = 0$ and $\sum_i c_i \bv(X_i) = 0$.
    	\item Every linear functional $\lambda$ on $\spn\{\bP_X \mid X \in \mathcal{B}\}$ has the form $\lambda(\bP_X) = \lambda_0 + \nu(X)$ with $\lambda_0 \in \R$ and $\nu : \mathcal{B} \to \R$ finitely additive, and every such pair $(\lambda_0, \nu)$ defines one.
    \end{enumerate}
\end{lemma}

\begin{proof}
    \noindent (1) The row of $\sum_i c_i \bP_{X_i}$ read off by $\bb_1$ is $\sum_i c_i \bv(X_i)$, and the row read off by $\bb_0$ is $(\sum_i c_i) \ind - \sum_i c_i \bv(X_i)$, where $\ind$ is the all-ones vector; both vanish exactly under the stated conditions. 
    
    \noindent (2) By (1), the linear map $\bP \mapsto (\kappa(\bP), r(\bP))$, with $\kappa$ the common column sum and $r$ the row read off by $\bb_1$, is injective on the span, and its image is $\R \oplus V$ with $V := \spn\{\bv(X) \mid X \in \mathcal{B}\}$, since $(1, 0)$ is the image of $\bP_\emptyset$ and $(0, \bv(X))$ that of $\bP_X - \bP_\emptyset$. A linear functional on $\R \oplus V$ is a pair $(\lambda_0, \lambda_1)$ with $\lambda_1 : V \to \R$ linear. Given $\lambda_1$, the set function $\nu(X) := \lambda_1(\bv(X))$ is finitely additive, because $\bv(X \sqcup Y) = \bv(X) + \bv(Y)$ for disjoint $X, Y \in \mathcal{B}$. Given a finitely additive $\nu$, the integral of $\mathcal{B}$-simple functions against $\nu$ is a well-defined linear functional on $V$, by the standard construction of the integral of simple functions against a finitely additive set function, and it sends $\bv(X)$ to $\nu(X)$.
\end{proof}

\begin{theorem}[Descent for functionals of sets]\label{thm:descentsets}
    Let $\mathcal{B}$ be a Boolean algebra of subsets of a set $D$, with $D \in \mathcal{B}$, and let $F : \mathcal{B} \to \{0, 1\}$. There is a linear $\Lambda : \Hom(\Free{D}, \Free{t}) \to \Free{t}$ with $\Lambda(\bP_X) = h_t(F(X))$ for all $X \in \mathcal{B}$ if and only if one of the following holds:
    \begin{enumerate}
    	\item $F$ is constant;
    	\item there is an ultrafilter $\mathcal{U}$ of $\mathcal{B}$ with $F(X) = 1$ exactly when $X \in \mathcal{U}$;
    	\item there is an ultrafilter $\mathcal{U}$ of $\mathcal{B}$ with $F(X) = 1$ exactly when $X \notin \mathcal{U}$.
    \end{enumerate}
    When $\mathcal{B}$ is the power set of a finite $D$, every ultrafilter is principal, and $F$ descends if and only if $F$ is constant, or $F(X) = [a_0 \in X]$, or $F(X) = [a_0 \notin X]$ for some $a_0 \in D$.
\end{theorem}

\begin{proof}
    \noindent Suppose $\Lambda$ exists; its coordinate along $\bb_1$ is a linear functional $\lambda$ on $\spn\{\bP_X\}$ with $\lambda(\bP_X) = F(X)$, so, by Lemma~\ref{lem:encodedsets}. $F(X) = \lambda_0 + \nu(X)$ with $\nu$ finitely additive; evaluating at $\emptyset$ gives $\lambda_0 = F(\emptyset)$ and $\nu = F - F(\emptyset)$. If $F(\emptyset) = 0$, then $\nu = F$ is a finitely additive set function with values in $\{0, 1\}$. Either $\nu(D) = 0$, in which case $0 \le \nu(X) \le \nu(D)$ forces $\nu \equiv 0$ and $F$ is constant; or $\nu(D) = 1$, in which case $\mathcal{U} := \{X \in \mathcal{B} : \nu(X) = 1\}$ is an ultrafilter: it contains $D$ and omits $\emptyset$; it is upward closed, since $\nu(Y) = \nu(X) + \nu(Y \setminus X) \ge \nu(X)$ for $X \subseteq Y$; it is closed under intersection, since $\nu(X \cap Y) = \nu(X) + \nu(Y) - \nu(X \cup Y) = 1$ for $X, Y \in \mathcal{U}$; and $\nu(X) + \nu(D \setminus X) = 1$ places exactly one of $X$ and its complement in $\mathcal{U}$. This is case 2. If $F(\emptyset) = 1$, the same argument applied to $1 - F$ gives case 1 or case 3.
    
    \noindent Conversely, in each of the three cases $F - F(\emptyset)$ is finitely additive, an ultrafilter indicator being a $\{0,1\}$-valued finitely additive set function, so Lemma~\ref{lem:encodedsets} provides $\lambda$ on the span with $\lambda(\bP_X) = F(X)$. Extend $\lambda$ and the column-sum functional $\kappa$ linearly from the span to $\Hom(\Free{D}, \Free{t})$ and set $\Lambda(\bP) := \lambda(\bP)\, \bb_1 + (\kappa(\bP) - \lambda(\bP))\, \bb_0$; then $\Lambda(\bP_X) = F(X)\, \bb_1 + (1 - F(X))\, \bb_0 = h_t(F(X))$. For a finite $D$, every ultrafilter of the power set is generated by a single point.
\end{proof}

\cite{Grefenstette2013} shows that universal and existential quantification fail to be multilinear on set vectors by scaling an argument; here, the failure persists for every linear map on the operator carrier. Descent also constrains functionals whose values lie in a primitive type other than $t$, when the argument domain is finite.

\begin{theorem}[Descent for finite argument domains]\label{thm:descentfinite}
    Let $\alpha = \langle \alpha_1, \alpha_2 \rangle$ with $\mathcal{D}_{\alpha_1}$ finite and $\alpha_2$ primitive, and let $F : \mathcal{D}_\alpha \to \mathcal{D}_t$. Then $F$ descends if and only if there are $a_0 \in \mathcal{D}_{\alpha_1}$ and $G : \mathcal{D}_{\alpha_2} \to \{0, 1\}$ with $F(f) = G(f(a_0))$ for all $f \in \mathcal{D}_\alpha$.
\end{theorem}

\begin{proof}
    For $a \in \mathcal{D}_{\alpha_1}$ and $u \in \Free{\alpha_2}$ let $E_a(u) \in \Hom(\Free{\alpha_1}, \Free{\alpha_2})$ be the linear map sending $\bb_a$ to $u$ and every other basis vector to $0$, so that $h_\alpha(f) = \sum_a E_a(\bb_{f(a)})$, a finite sum. If $\Lambda$ exists, then $\Lambda(h_\alpha(f)) = \sum_a \Lambda(E_a(\bb_{f(a)}))$, and reading the coordinate along $\bb_1$ gives $F(f) = \sum_a g_a(f(a))$ for real-valued functions $g_a$ on $\mathcal{D}_{\alpha_2}$. Suppose two positions $a \neq a'$ had nonconstant $g_a$ and $g_{a'}$: choose $x, x'$ with $g_a(x) \neq g_a(x')$, $y, y'$ with $g_{a'}(y) \neq g_{a'}(y')$, and any $f$. The values of $F$ at the four functions $f[a \mapsto u, a' \mapsto v]$, $u \in \{x, x'\}$, $v \in \{y, y'\}$, are $c$, $c + \delta$, $c + \epsilon$, $c + \delta + \epsilon$ with $\delta, \epsilon \neq 0$; membership of $c$, $c + \delta$, $c + \epsilon$ in $\{0, 1\}$ forces $\delta = \epsilon = 1 - 2c$, and then $c + \delta + \epsilon = 2 - 3c \notin \{0, 1\}$. Hence at most one $g_a$ is nonconstant, and $F(f) = c' + g_{a_0}(f(a_0)) =: G(f(a_0))$. Conversely, given $a_0$ and $G$, let $\tilde{G} : \Free{\alpha_2} \to \Free{t}$ be the free lift of $G$ and set $\Lambda(L) := \tilde{G}(L\, \bb_{a_0})$; $\Lambda$ is linear, and $\Lambda(h_\alpha(f)) = \tilde{G}(\bb_{f(a_0)}) = h_t(G(f(a_0))) = h_t(F(f))$.
\end{proof}

The two descent theorems overlap on finite power sets, where they agree; Theorem~\ref{thm:descentsets} covers infinite index domains and the $\sigma$-algebras of measured sorts, and Theorem~\ref{thm:descentfinite} covers functionals of individual concepts and of other primitive-valued functions on finite frames. 

\begin{corollary}[Individuals]\label{cor:individuals}
    Let $\alpha = \langle \langle e, t \rangle, t \rangle$ with $\mathcal{D}_e$ finite. The functionals of $\mathcal{D}_{\langle e, t \rangle}$ that descend are the constants, Montague's individuals $\lambda P.\, P(d)$, and their negations $\lambda P.\, \neg P(d)$, for $d \in \mathcal{D}_e$. For infinite $\mathcal{D}_e$, the descending functionals are the constants, the indicators of ultrafilters on $\mathcal{D}_e$, and the complements of these, the principal ultrafilters giving the individuals.
\end{corollary}

\begin{proof}\label{prf:individuals}
    Theorem~\ref{thm:descentsets} with $\mathcal{B}$ the power set of $\mathcal{D}_e$.
\end{proof}

\begin{corollary}[Determiners]\label{cor:determiners}
    Let $\Delta : \mathcal{D}_{\langle e, t \rangle} \to \mathcal{D}_{\langle \langle e, t \rangle, t \rangle}$ be the denotation of a determiner, let $P \subseteq \mathcal{D}_e$ be a restrictor, and write $F_P(Q) := \Delta(P)(Q)$ for the scope functional. If $P$ has at least two elements, then $F_P$ does not descend for $\Delta$ among \textit{every}, \textit{some}, \textit{no}, \textit{most}, \textit{at least $n$} with $1 \le n \le |P|$, and \textit{exactly $n$} with $0 \le n \le |P|$. If $P = \{d\}$, then \textit{every}, \textit{some}, \textit{at least one}, and \textit{exactly one} reduce to $\lambda Q.\, Q(d)$, and \textit{no} and \textit{exactly zero} to $\lambda Q.\, \neg Q(d)$, all of which descend; if $P$ is empty, each of these determiners yields a constant scope functional, which descends.
\end{corollary}

\begin{proof}
    By Theorem~\ref{thm:descentsets}, a nonconstant $F_P$ descends only if $\mathcal{U} := \{Q : F_P(Q) = 1\}$ or its complement family is an ultrafilter, which contains exactly one member of each complementary pair $\{Q, \mathcal{D}_e \setminus Q\}$, is upward closed, and is closed under intersection. Fix distinct $p_1, p_2 \in P$. 

    \begin{itemize}
        \item For \textit{every}, $\mathcal{U} = \{Q \supseteq P\}$ contains neither $\{p_1\}$ nor its complement. 
        \item For \textit{some}, $\mathcal{U} = \{Q : Q \cap P \neq \emptyset\}$ contains both $\{p_1\}$ and its complement, which contains $p_2$.
        \item For \textit{no}, the complement family does. 
        \item For \textit{at least $n$}, with $n = |P|$, the determiner is \textit{every}; for $1 \le n \le |P| - 1$ choose $Q, Q' \subseteq P$ of size $n$ with $p_1 \in Q \setminus Q'$ and $p_2 \in Q' \setminus Q$: both lie in $\mathcal{U}$, their intersection has fewer than $n$ elements and does not, and the complement family contains $\emptyset$ but omits $\mathcal{D}_e$. 
        \item For \textit{exactly $n$}, the case $n = 0$ is \textit{no} and the case $n = |P|$ is \textit{every}; for $1 \le n < |P|$, $\mathcal{U}$ contains a set of $n$ elements of $P$ but omits its superset $\mathcal{D}_e$, and the complement family contains $\emptyset$. 
        \item For \textit{most}, the case $|P| = 2$ is \textit{every}; for $|P| \ge 3$ enumerate $P$ and let $Q$ consist of the first $\lfloor |P|/2 \rfloor + 1$ elements and $Q'$ of the last $\lfloor |P|/2 \rfloor + 1$: both lie in $\mathcal{U}$, their intersection has at most two elements and at most $|P|/2$, so it does not, and the complement family contains $\emptyset$ but omits $\mathcal{D}_e$. 
    \end{itemize}
    
     \noindent In every case $F_P$ is nonconstant. The singleton and empty cases are read off from the definitions.
\end{proof}

\begin{remark}[Attitudes]\label{rem:attitudes}
    An attitude operator such as \textit{believes} evaluated at an index $s$ and for an agent $a$ is a necessity operator over the alternatives $\mathrm{Dox}_a(s)$; we see, Proposition~\ref{prop:modaldescent} below applies to it word for word: it acts linearly on encoded propositions only when the agent has, up to null sets, a single alternative. In this way, we can see precisely that the failure of descent is the fact that attitude verbs take propositions, whose very encodings are operators, as arguments.
\end{remark}

\section{Modal operators}\label{sec:modalops}

We now address the representation of the modal operators $\Box$ (necessity) and $\Diamond$ (possibility): propositions become operators $\Free{s} \to \Free{t}$, or the vectors of truth values; accessibility relations become kernels; the modal quantifiers are realized by a decision on an accumulated value. To reduce notation, we work over a single sort, and write $(D, \mathcal{M}, \mathcal{R}, \mu)$ for its measure frame, treating a proposition as a measurable function $\varphi : D \to \{0, 1\}$, which we identify with its truth set. The compound case acts on the $\sigma$-coordinate with the others fixed.

\begin{definition}[Accumulation]\label{def:accumulation}
    For a nonnegative measurable function $\varphi$ on $D$ and $w \in D$, the \textbf{accumulation} of $\varphi$ at $w$ is
    $$(\bA \varphi)(w) := \int_{\mathcal{R}(w)} \varphi(w')\, d\mu_w(w') = \mu_w\bigl(\mathcal{R}(w) \cap \varphi\bigr) \quad \text{for } \varphi \in \{0,1\}^D,$$
    with values in $[0, \infty]$. The map $\varphi \mapsto \bA\varphi$ is additive and positively homogeneous, and linear on measurable functions with finite accumulation.
\end{definition}

In terms of the accumulation, clauses 5 and 6 of Definition~\ref{def:intendenot} now become
\begin{equation}\label{eq:box}
	(\Box \varphi)(w) = 1 \iff (\bA (\ind - \varphi))(w) = 0, \qquad (\Diamond \varphi)(w) = 1 \iff (\bA \varphi)(w) > 0,
\end{equation}
a linear accumulation followed by a decision. For a discrete sort, $\bA$ is the adjacency matrix $\bA_{ij} = 1 \iff \mathcal{R}(w_i, w_j)$, and $(\bA \bv(\varphi))_i = \sum_j \bA_{ij}\, \varphi(w_j)$ counts the accessible worlds from $w_i$ that satisfy $\varphi$, so that $\Box \varphi$ holds at $w_i$ when the count of accessible $\neg\varphi$-worlds is zero and $\Diamond \varphi$ when the count of accessible $\varphi$-worlds is at least one. When the out-degree $\deg_{\bA}(w_i) = \sum_j \bA_{ij}$ is finite, the necessity condition takes the equivalent form
\begin{equation}\label{eq:boxdeg}
	(\Box \varphi)(w_i) = 1 \iff (\bA \bv(\varphi))_i = \deg_{\bA}(w_i);
\end{equation}
for infinite out-degree, the sums take values in $[0, \infty]$ and \eqref{eq:box} as a world accessing infinitely many $\varphi$-worlds, and infinitely many $\neg\varphi$-worlds satisfies $\Diamond\varphi$ and fails $\Box\varphi$. For a measured sort, $\bA$ is an integral kernel, and \eqref{eq:box} says that $\varphi$ holds $\mu_w$-almost everywhere on $\mathcal{R}(w)$, and on a subset of $\mathcal{R}(w)$ of positive measure.

\begin{example}[Four-world frame]\label{ex:necfin}
    Consider the frame $\mathcal{F} = (\{w_1, w_2, w_3, w_4\}, \mathcal{R})$ with
    $$\mathcal{R} = \{(w_1, w_2), (w_1, w_4), (w_2, w_3), (w_3, w_2), (w_4, w_4)\}.$$
    From $w_1$ we can access $w_2$ and $w_4$; from $w_2$ we can access $w_3$; from $w_3$ we can access $w_2$; from $w_4$ we can access only itself. The adjacency matrix is
    $$\bA = \begin{pmatrix} 0 & 1 & 0 & 1 \\ 0 & 0 & 1 & 0 \\ 0 & 1 & 0 & 0 \\ 0 & 0 & 0 & 1 \end{pmatrix},$$
    where row $i$, column $j$ indicates whether $w_i$ can access $w_j$. Suppose $\bv(\varphi) = (1, 1, 0, 1)^\top$. Then $\bA \bv(\varphi) = (2, 0, 1, 1)^\top$ against out-degrees $(2, 1, 1, 1)$, so $\bv(\Box\varphi) = (1, 0, 1, 1)^\top$ by \eqref{eq:boxdeg}, and $\bv(\Diamond\varphi) = (1, 0, 1, 1)^\top$ by the positivity test. The two coincide here because every accessible set is uniformly $\varphi$ or uniformly $\neg\varphi$; $\varphi$ is necessary and possible at $w_1$, $w_3$, and $w_4$, and neither at $w_2$, whose only accessible world is $w_3$, where $\varphi$ fails.
\end{example}

\begin{example}[Infinite chain]\label{ex:neccount}
    Consider a countably infinite set of worlds $\{w_0, w_1, w_2, \ldots\}$ with accessibility forming a one-way infinite chain, $\mathcal{R}(w_i, w_j)$ if and only if $j = i + 1$. Each world accesses only its immediate successor, so $(\bA\varphi)(w_i) = \varphi(w_{i+1})$ and
    $$(\Box \varphi)(w_i) = (\Diamond \varphi)(w_i) = \varphi(w_{i+1}).$$
    Suppose $\varphi$ holds only at $w_5$, $w_{11}$, and $w_{19}$. Then $\Box\varphi$ holds at $w_4$, $w_{10}$, and $w_{18}$, and nowhere else; in particular $(\Box\varphi)(w_5) = 0$ although $\varphi(w_5) = 1$, because $\varphi(w_6) = 0$. Necessity at $w_i$ depends entirely on the truth of $\varphi$ at the accessible worlds. Over a singleton accessible set, necessity and possibility coincide, in which both read the proposition at a single index; Proposition~\ref{prop:modaldescent} shows that this is the one configuration in which a modal operator acts linearly on encoded propositions.
\end{example}

\begin{example}[Continuous time]\label{ex:necunc}
    Continuous worlds are non-intuitive, so for illustration we take time as the sort, and a simple machine as proxy for truth in a model. Let $D = \R$ with Lebesgue measure $\mu$, and let $\mathcal{R}(t, t')$ hold if and only if $|t' - t| \le 1$; each time then accesses the interval $\mathcal{R}(t) = [t - 1, t + 1]$ and $\mu(\mathcal{R}(t)) = 2$. Let $\varphi$ indicate whether the machine is running at time $t'$. Then
    $$(\Box \varphi)(t) = 1 \iff \int_{t-1}^{t+1} \varphi(t')\, dt' = 2,$$ 
    which says that the machine runs at almost every moment within one unit of $t$. Suppose $\varphi = 1$ on $[0, 10]$ except on the Cantor set $C \subseteq [3, 5]$, where $\varphi = 0$; $C$ is uncountable and $\mu(C) = 0$, so at $t = 4$ the integral over $[3, 5]$ equals $\mu([3, 5] \setminus C) = 2$ and $(\Box\varphi)(4) = 1$: the box holds despite uncountably many failures, because those failures have measure zero. Suppose instead that $\varphi = 0$ on a measurable $E \subseteq [3, 5]$ with $\mu(E) = 0.5$; then the integral is $1.5$ and $(\Box\varphi)(4) = 0$: the box fails although $\varphi$ holds at most times in the interval, because the exceptions have positive measure. Dually, if the machine runs only on $[4.25, 4.35]$, then $(\Diamond\varphi)(4) = 1$, since $\mu([4.25, 4.35]) = 0.1 > 0$, while a machine running only at the single instant $4.3$, or only on $C$, has $(\Diamond\varphi)(4) = 0$: possibility requires truth on a set of positive measure.
\end{example}

The measure-theoretic clause generalizes universal quantification to continua, with ``for all'' becoming ``for $\mu$-almost all'' and ``there exists'' becoming ``on a set of positive measure''. Which natural-language expressions, if any, quantify in this way is a question of Section~\ref{sec:graded}; let us here develop the logic first.

\subsection{Normality and correspondence}\label{subsec:correspondence}

The measure-theoretic box is a neighborhood operator, whose neighborhoods at $w$ are the subsets of $D$, whose complements within $\mathcal{R}(w)$ are $\mu_w$-null. These neighborhoods form a filter, which accounts for the elementary logic.

\begin{theorem}[Normality]\label{thm:normal}
    On every measure frame, $\Box$ and $\Diamond$ are dual, $\Diamond\varphi = \neg\Box\neg\varphi$; the operator $\Box$ is monotone, validates $\Box(\varphi \to \psi) \to (\Box\varphi \to \Box\psi)$, and validates $\Box\varphi$ whenever $\varphi$ is valid. The set of formulas valid on any class of measure frames is therefore a normal modal logic.
\end{theorem}

\begin{proof}\label{prf:normal}
    Duality is the definition of $\Diamond$. For the rest, fix $w$ and write $N_w := \{ X \in \mathcal{M} : \mu_w(\mathcal{R}(w) \setminus X) = 0 \}$, so that $\Box\varphi$ holds at $w$ if and only if the truth set of $\varphi$ lies in $N_w$. Since a subset of a null set is null and the union of two null sets is null, $N_w$ is upward closed and closed under finite intersection, and it contains $D$. Monotonicity is upward closure. If the truth sets of $\varphi \to \psi$ and $\varphi$ lie in $N_w$, then so does their intersection, which is contained in the truth set of $\psi$, and $\Box\psi$ holds at $w$. If $\varphi$ is valid, its truth set is $D \in N_w$ at every $w$.
\end{proof}

The correspondence theory changes under measures, which separates the frame conditions that refer to a single point from those that refer to inclusions of accessible sets. Validity of an axiom on a frame means truth at every index under every assignment of measurable truth sets to the propositional letters.

\begin{proposition}[Correspondence]\label{prop:correspondence}
    Let $(D, \mathcal{M}, \mathcal{R}, \mu)$ be a measure frame.
    \begin{enumerate}
    	\item[(D)] $\Box\varphi \to \Diamond\varphi$ is valid if and only if $\mu_w(\mathcal{R}(w)) > 0$ for every $w$.
    	\item[(T)] $\Box\varphi \to \varphi$ is valid if and only if, for every $w$, $w \in \mathcal{R}(w)$ and $\mu_w(\{w\}) > 0$.
    	\item[(B)] $\varphi \to \Box\Diamond\varphi$ is valid if and only if, for every $w$, $\mu_w$-almost every $w' \in \mathcal{R}(w)$ satisfies $w \in \mathcal{R}(w')$ and $\mu_{w'}(\{w\}) > 0$.
    	\item[(4)] If $\mathcal{R}$ is transitive and, for all $w$ and all $w' \in \mathcal{R}(w)$, every $\mu_w$-null subset of $\mathcal{R}(w')$ is $\mu_{w'}$-null, then $\Box\varphi \to \Box\Box\varphi$ is valid.
    	\item[(5)] If $\mathcal{R}$ is euclidean and, for all $w$ and all $w' \in \mathcal{R}(w)$, every $\mu_{w'}$-null subset of $\mathcal{R}(w)$ is $\mu_w$-null, then $\Diamond\varphi \to \Box\Diamond\varphi$ is valid.
    \end{enumerate}
    Under counting measure, (D), (T), and (B) are seriality, reflexivity, and symmetry, and the hypotheses of (4) and (5) reduce to transitivity and euclideanness. When all $\mu_w$ coincide with a single measure $\mu$, the absolute-continuity hypotheses of (4) and (5) hold automatically.
\end{proposition}

\begin{proof}\label{prf:correspondence}

    \noindent We take each in turn.
    
    \noindent (D) If $\mu_w(\mathcal{R}(w)) > 0$ and $\Box\varphi$ holds at $w$, then, since $\mu_w(\mathcal{R}(w)) = \mu_w(\mathcal{R}(w) \cap \varphi) + \mu_w(\mathcal{R}(w) \setminus \varphi)$ and the second summand is zero, $\mu_w(\mathcal{R}(w) \cap \varphi) = \mu_w(\mathcal{R}(w)) > 0$, so $\Diamond\varphi$ holds. If $\mu_w(\mathcal{R}(w)) = 0$, then $\Box\bot$ holds at $w$ and $\Diamond\bot$ fails.
    
    \noindent (T) If $w \in \mathcal{R}(w)$ and $\mu_w(\{w\}) > 0$, then $\varphi$ false at $w$ gives $\mu_w(\mathcal{R}(w) \setminus \varphi) \ge \mu_w(\{w\}) > 0$, so $\Box\varphi$ fails. Otherwise $\mu_w(\mathcal{R}(w) \cap \{w\}) = 0$; take $\varphi$ with truth set $D \setminus \{w\}$, which is measurable: then $\mathcal{R}(w) \setminus \varphi = \mathcal{R}(w) \cap \{w\}$ is null, $\Box\varphi$ holds at $w$, and $\varphi$ fails there.
    
    \noindent (B) Suppose the condition holds and $\varphi$ is true at $w$. For $\mu_w$-almost every $w' \in \mathcal{R}(w)$, the set $\mathcal{R}(w') \cap \varphi$ contains $w$ and has positive $\mu_{w'}$-measure, so $\Diamond\varphi$ holds at $w'$; the set of $w' \in \mathcal{R}(w)$ where it fails is $\mu_w$-null, and $\Box\Diamond\varphi$ holds at $w$. Conversely, take $\varphi$ with truth set $\{w\}$: $\Box\Diamond\varphi$ at $w$ says precisely that $\mu_{w'}(\mathcal{R}(w') \cap \{w\}) > 0$ for $\mu_w$-almost every $w' \in \mathcal{R}(w)$.
    
    \noindent (4) If $\Box\varphi$ holds at $w$, the set $N := \mathcal{R}(w) \setminus \varphi$ is $\mu_w$-null. For $w' \in \mathcal{R}(w)$, transitivity gives $\mathcal{R}(w') \subseteq \mathcal{R}(w)$, so $\mathcal{R}(w') \setminus \varphi \subseteq N$ is $\mu_w$-null and, by hypothesis, $\mu_{w'}$-null: $\Box\varphi$ holds at every $w' \in \mathcal{R}(w)$, and $\Box\Box\varphi$ holds at $w$.
    
    \noindent (5) If $\Diamond\varphi$ holds at $w$, then $\mathcal{R}(w) \cap \varphi$ has positive $\mu_w$-measure. For $w' \in \mathcal{R}(w)$, euclideanness gives $\mathcal{R}(w) \subseteq \mathcal{R}(w')$, so $\mathcal{R}(w') \cap \varphi \supseteq \mathcal{R}(w) \cap \varphi$, which by hypothesis has positive $\mu_{w'}$-measure: $\Diamond\varphi$ holds at every $w' \in \mathcal{R}(w)$, and $\Box\Diamond\varphi$ holds at $w$.
    The counting-measure and single-measure readings are immediate.
\end{proof}

Reflexivity and symmetry lose their axioms under an atomless measure, because a single point carries no mass; transitivity and euclideanness keep theirs, because inclusions of accessible sets transfer null sets. Seriality corresponds to positive mass of the accessible set. The converse directions of (4) and (5) weaken to transitivity and euclideanness up to null sets, which we leave open here.

\begin{corollary}[Almost everywhere]\label{cor:ae}
    Let $\mathcal{R}$ be the universal relation and let all $\mu_w$ coincide with a single $\sigma$-finite atomless measure $\mu$ with $\mu(D) > 0$. Then the frame validates K, D, 4, and 5 and refutes T and B; $\Box\varphi$ holds at some index if and only if it holds at every index, and it holds if and only if $\varphi$ is true $\mu$-almost everywhere.
\end{corollary}

\begin{proof}\label{prf:ae}
    K is Theorem~\ref{thm:normal}; D, 4, and 5 are Proposition~\ref{prop:correspondence}, the universal relation being transitive and euclidean and the single measure supplying absolute continuity; T and B are refuted by (T) and (B), since $\mu(\{w\}) = 0$ for every $w$. The last claim is the definition with $\mathcal{R}(w) = D$.
\end{proof}

The logic of ``almost surely'' is, thus, KD45, on the universal frame, the logic of belief with the interpretation that the excluded worlds form a null set.

\subsection{Descent of modal operators}\label{subsec:modaldescent}

A modal operator at a fixed index is a functional of propositions; we have seen this. Theorem~\ref{thm:descentsets} decides when it acts linearly on encoded propositions.

\begin{proposition}[Modal descent]\label{prop:modaldescent}
    Let $(D, \mathcal{M}, \mathcal{R}, \mu)$ be a measure frame and $w \in D$. The functional $X \mapsto (\Box X)(w)$ on $\mathcal{M}$ descends if and only if $\mu_w(\mathcal{R}(w)) = 0$ or the restriction of $\mu_w$ to $\mathcal{R}(w)$ is an atom, in the sense that every $X \in \mathcal{M}$ has $\mu_w(\mathcal{R}(w) \cap X) = 0$ or $\mu_w(\mathcal{R}(w) \setminus X) = 0$. The same holds for $X \mapsto (\Diamond X)(w)$. For a discrete sort, both descend if and only if $|\mathcal{R}(w)| \le 1$.
\end{proposition}

\begin{proof}
    Let $F(X) := (\Box X)(w) = [\mu_w(\mathcal{R}(w) \setminus X) = 0]$. If $\mu_w(\mathcal{R}(w)) = 0$, then $F \equiv 1$, which is case 1 of Theorem~\ref{thm:descentsets}. Otherwise $F(\emptyset) = 0$ and $F(D) = 1$, so $F$ descends if and only if $\mathcal{U} := \{X \in \mathcal{M} : \mu_w(\mathcal{R}(w) \setminus X) = 0\}$ is an ultrafilter of $\mathcal{M}$. It is a filter by the proof of Theorem~\ref{thm:normal}, and it is an ultrafilter if and only if for each $X \in \mathcal{M}$ one of $\mathcal{R}(w) \setminus X$ and $\mathcal{R}(w) \cap X$ is null, which is the atom condition. For $\Diamond$, $(\Diamond X)(w) = 1 - (\Box (D \setminus X))(w)$, and $\bP_{D \setminus X} = N \bP_X$ with $N$ the swap matrix, so a linear $\Lambda$ for $\Box$ yields the linear $\bP \mapsto N \Lambda(N \bP)$ for $\Diamond$ and conversely. For counting measure, an atom of $\mathcal{R}(w)$ is a singleton, since any $x \in \mathcal{R}(w)$ has $\mu_w(\{x\}) = 1$ and $\mathcal{R}(w) \setminus \{x\}$ has positive measure unless it is empty.
\end{proof}

Proposition~\ref{prop:modaldescent} delimits the linear regime of modality: the next-time operator of Example~\ref{ex:neccount}, its constant relatives, and their measure-theoretic analogues over a single atom. Every box over two or more accessible indices, and every box over an atomless measure of positive mass, needs the decision step of \eqref{eq:box}. Since $\Box$ at $w$ is the determiner \textit{every} with restrictor $\mathcal{R}(w)$, this is essentially Corollary~\ref{cor:determiners}, just from another point of view.

\begin{remark}[Boolean semiring]\label{rem:semiring}
    Over the semiring $(\{0, 1\}, \vee, \wedge)$, the product $(\bA \odot \bv)_w = \bigvee_j (\bA_{wj} \wedge \bv_j)$ is the indicator of $\Diamond\varphi$ and is linear over that semiring, with $\Box$ its dual. The decision step in \eqref{eq:box} is the what happens when we work over $\R$ instead; the real accumulation then carries counts and measures, which we use in Section~\ref{sec:graded}.
\end{remark}

\begin{remark}[Tree of numbers]\label{rem:tree}
    Under counting measure with $\mathcal{R}(w)$ finite, $\Box$ at $w$ is the determiner \textit{every} and $\Diamond$ the determiner \textit{some}, each with restrictor $\mathcal{R}(w)$ and scope $\varphi$, which is the sense in which modals quantify over worlds \cite{Kratzer1991}. A determiner meaning that is conservative, extension-invariant, and isomorphism-invariant depends only on the pair $(|A \setminus B|, |A \cap B|)$ of restrictor $A$ and scope $B$ \cite{vanBenthem1986,KeenanStavi1986}, so any such modal determiner over $\mathcal{R}(w)$ is a decision on the pair of accumulations $\bigl((\bA(\ind - \bv))_w, (\bA \bv)_w\bigr)$: \textit{most} compares them, \textit{at least $n$} thresholds the second. The form of \eqref{eq:box} is thus the general form of quantification over indices, and, again, Section~\ref{sec:graded} continues it with measures in place of counts.
\end{remark}

\subsection{Other intensional index sorts}\label{subsec:othersorts}

The framework developed above applies uniformly to any index sort $\sigma \in \Sigma$. Beyond possible worlds, two natural index sorts are temporal indices and spatial indices. Temporal logic has a history beginning with Prior's foundational work \cite{prior1957time} and developed extensively in subsequent decades \cite{Venema2008,ploug2012branching}. Spatial logics have received similar treatment \cite{Kamide2005,Simons2006,vanbenthem2007modal}, and the interaction of spatial and temporal modalities raises additional complexities \cite{kontchakov2007spatial,Chierchia2000}. We briefly survey how the operator construction specializes to these cases.

Let $\iota \in \Sigma$ be the temporal index sort, with domain $\mathcal{D}_\iota$ representing times. The choice of domain structure depends on the application: $\mathbb{Z}$ or $\N$ for computation steps, or a continuum in $\R$ \cite{HIRSHFELD20121}; tree structures for branching indeterminate futures \cite{Andreoletti2024}; intervals of the form $\{[a, b] : a \le b\}$ for event durations \cite{allen1997actions}. The accessibility relation $\mathcal{R}_\iota$ encodes which times are reachable or relevant from a given time, and several natural choices correspond to familiar temporal operators. Strict future accessibility, $\mathcal{R}_\iota(t, t') \Leftrightarrow t' > t$, yields $\Box_\iota\varphi$ at $t$ meaning ``$\varphi$ holds at all future times'', the operator $\mathbf{G}$ of temporal logic, and $\Diamond_\iota$ the operator $\mathbf{F}$. Bounded future accessibility, $\mathcal{R}_\iota(t, t') \Leftrightarrow 0 < t' - t \le \delta$, restricts this to a finite horizon. Immediate successor accessibility in discrete time, $\mathcal{R}_\iota(t, t') \Leftrightarrow t' = t + 1$, reduces $\Box_\iota\varphi$ at $t$ to $\varphi(t + 1)$, the operator $\mathbf{X}$, which by Proposition~\ref{prop:modaldescent} is the temporal operator that acts linearly on encoded propositions. Temporal proximity, $\mathcal{R}_\iota(t, t') \Leftrightarrow |t' - t| \le \epsilon$, captures a form of temporal stability. Past-directed accessibility, $\mathcal{R}_\iota(t, t') \Leftrightarrow t' < t$, yields the past-tense modality ``it has always been the case that $\varphi$''. For discrete $\mathcal{D}_\iota$ the adjacency matrix $\bA^\iota$ encodes temporal accessibility and the modal operators are computed as in \eqref{eq:box}; for continuous $\mathcal{D}_\iota = \R$ the accumulation is an integral transform, and for bounded future accessibility with $\mathcal{R}_\iota(t) = (t, t + \delta]$ necessity becomes
$$(\Box_\iota\varphi)(t) = 1 \iff \int_t^{t + \delta} \varphi(t')\, dt' = \delta,$$
the statement that $\varphi$ holds at almost every time in the coming interval of length $\delta$. The operator $\mathbf{U}$ (``$\varphi$ until $\psi$'') requires a construction involving two propositions and path structure, which extends beyond the simple modal operators treated here.

Let $\ell \in \Sigma$ be the spatial index sort, with domain $\mathcal{D}_\ell$ representing locations \cite{Kamide2005,Simons2006}. Again, the structure varies by application: discrete finite on $\{L_1, \ldots, L_n\}$; graph-based vertices \cite{cardelli2002spatial,cardelli2000anytime}; metric and topological spaces \cite{Rescher1968}. The accessibility relation $\mathcal{R}_\ell$ encodes spatial relationships. Metric proximity, $\mathcal{R}_\ell(x, y) \Leftrightarrow d(x, y) \le r$ for a metric $d$ and radius $r$, yields $\Box_\ell\varphi$ at $x$ meaning ``$\varphi$ holds everywhere within distance $r$ of $x$''. Graph adjacency, $\mathcal{R}_\ell(x, y) \Leftrightarrow (x, y) \in E$ for an edge set $E$, gives ``$\varphi$ holds at all neighbors of $x$''; reachability extends this to the transitive closure; containment relations, $\mathcal{R}_\ell(x, y) \Leftrightarrow y \in \mathrm{Region}(x)$, model ``everywhere in this region''; visibility, where $\mathcal{R}_\ell(x, y)$ holds if $y$ is visible from $x$ without obstruction, is relevant for robotics, surveillance, and sensor coverage \cite{morales2007new,walega2019}. For discrete locations the adjacency matrix $\bA^\ell$ is the graph adjacency matrix, or its transitive closure for reachability. For continuous $\mathcal{D}_\ell = \R^n$ with metric proximity, the accumulation is $(\bA^\ell\varphi)(x) = \int_{B_r(x)} \varphi(y)\, d\mu(y)$ over the ball of radius $r$ about $x$, and necessity becomes
$$(\Box_\ell\varphi)(x) = 1 \iff \int_{B_r(x)} \varphi(y)\, dy = \mu(B_r(x)) = \frac{\pi^{n/2}}{\Gamma(n/2 + 1)}\, r^n,$$
the volume of the $n$-dimensional ball, generalizing the familiar cases $2r$ for $n = 1$, $\pi r^2$ for $n = 2$, $\tfrac{4}{3}\pi r^3$ for $n = 3$, and so on.

\begin{remark}[Topological semantics]\label{rem:topological}
    Spatial modalities are also treated by the topological semantics for modal logic \cite{vanbenthem2007modal}, where $\Box\varphi$ holds at $x$ if and only if $\varphi$ holds on some open neighborhood of $x$, the interior operator, and $\Diamond$ is closure. The measure-theoretic box differs from the interior operator version: a set of full measure in the ball about $x$ need contain no open neighborhood of $x$, and a set containing a small open neighborhood of $x$ need have nothing like full measure in a ball of fixed radius. The relation between the two semantics on the Lebesgue measure algebra is worked out in \cite{Lando2012}.
\end{remark}

\subsection{Products}\label{subsec:products}

With multiple index sorts $\Sigma = \{w, \iota, \ell\}$, the compound index space is $S = \mathcal{D}_w \times \mathcal{D}_\iota \times \mathcal{D}_\ell$ \cite{KURUCZ2007869,gabbay2003many,reynolds2001products}. A compound index $s = (w, t, x) \in S$ specifies a possible world, a time, and a location. The intension of a proposition is a function $S \to \{0, 1\}$, and its vector representation is a linear map $\Free{s} \to \Free{t}$. When the accessibility relations of the sorts are independent, in the sense of Definition~\ref{def:intendenot}, each operator acts on its own dimension, $\Box_w$ quantifying over accessible worlds while holding time and location fixed, and likewise for $\Box_\iota$ and $\Box_\ell$; the tensor decomposition of $\Free{s}$ makes this a statement about matrices.

\begin{proposition}[Kronecker form]\label{prop:kronecker}
    Let all sorts be discrete, and let $\bA_\sigma$ be the adjacency matrix of $\mathcal{R}_\sigma$ on $\mathcal{D}_\sigma$. With respect to the basis $\bb_s = \bigotimes_{\sigma'} \bb_{s_{\sigma'}}$ of $\Free{s}$, the accessibility relation of $\Box_\sigma$ on $S$, which relates $s$ to $s[\sigma \mapsto w']$ for $w' \in \mathcal{R}_\sigma(s_\sigma)$, has adjacency matrix
    $$\bA_\sigma^{(S)} = \bigotimes_{\sigma' \in \Sigma} \bA^{(\sigma')}, \qquad \bA^{(\sigma')} = \begin{cases} \bA_\sigma & \sigma' = \sigma \\ I & \sigma' \neq \sigma. \end{cases}$$
    Consequently $\bA_\sigma^{(S)} \bA_{\sigma'}^{(S)} = \bA_{\sigma'}^{(S)} \bA_\sigma^{(S)}$ for all $\sigma, \sigma'$, over $\R$ and over the Boolean semiring alike, and the commutation axioms
    $$\Diamond_\sigma \Diamond_{\sigma'} \varphi \leftrightarrow \Diamond_{\sigma'} \Diamond_\sigma \varphi, \qquad \Box_\sigma \Box_{\sigma'} \varphi \leftrightarrow \Box_{\sigma'} \Box_\sigma \varphi$$
    are valid on every compound frame with independent sorts.
\end{proposition}

\begin{proof}[Proof sketch]
    The entry of $\bA_\sigma^{(S)}$ at $(s, s')$ is $1$ if and only if $s'_\sigma \in \mathcal{R}_\sigma(s_\sigma)$ and $s'_{\sigma'} = s_{\sigma'}$ for $\sigma' \neq \sigma$, which is the product of the corresponding entries of $\bA_\sigma$ and of the identity matrices, the entry of the Kronecker product. The mixed-product property of Kronecker products, which holds over any semiring, gives $\bA_\sigma^{(S)} \bA_{\sigma'}^{(S)} = \bigotimes_{\sigma''} \bigl( \bA^{(\sigma'')} \bB^{(\sigma'')} \bigr)$, where $\bA^{(\sigma'')}$ and $\bB^{(\sigma'')}$ are the factors of $\bA_\sigma^{(S)}$ and $\bA_{\sigma'}^{(S)}$ at $\sigma''$; for $\sigma \neq \sigma'$ at least one factor in each product is an identity, so the two orders agree factor by factor, and for $\sigma = \sigma'$ there is nothing to show. Over the Boolean semiring, $\bA_\sigma^{(S)} \odot \bv$ is the truth profile of $\Diamond_\sigma\varphi$ of Remark~\ref{rem:semiring}, so the commutation of the matrices is the first axiom, and duality gives the second.
\end{proof}

The commutation axioms are the characteristic axioms of product logics \cite{gabbay2003many}, whose third axiom, $\Diamond_\sigma \Box_{\sigma'} \varphi \to \Box_{\sigma'} \Diamond_\sigma \varphi$, is, likewise, valid on product frames. For measured sorts, the same commutation holds for the kernels, $(\bA_\sigma \bA_{\sigma'} \varphi)(s) = \int\!\!\int \varphi(s[\sigma \mapsto w'][\sigma' \mapsto w''])\, d\mu_{\sigma, s_\sigma}(w')\, d\mu_{\sigma', s_{\sigma'}}(w'')$ in either order by Tonelli's theorem, since the kernel of each sort depends on the coordinate of that sort alone. Compositions such as $\Box_w \Box_\iota\varphi$ (``necessarily, always $\varphi$'') or $\Diamond_\ell \Box_\iota\varphi$ (``somewhere, always $\varphi$'') then express complex modal claims by combining quantification across dimensions, in either order.

Independence fails when accessibility in one sort depends on position in another \cite{kontchakov2007spatial}. In relativistic spacetime, accessibility is constrained by the light cone: from $(t, x)$, the accessible points $(t', x')$ satisfy $|x' - x| \le c\,(t' - t)$ for $t' > t$, where $c$ is the speed of light \cite{Hirsch2018,Lewis2026,andreka2007logic}. The admissible locations depend on the elapsed time, so the cone relation on $\mathcal{D}_\iota \times \mathcal{D}_\ell$ is not a product $\mathcal{R}_\iota \times \mathcal{R}_\ell$ of a temporal and a spatial relation, its matrix is not a Kronecker product, and the cone modality is not a composite of a temporal box and a spatial box. Expanding reachability, in which the set of accessible locations grows with time as in exploration or diffusion, and world-dependent geography, in which different possible worlds carry different spatial accessibility structures, are of the same kind. Such dependencies are encoded by defining a relation and a kernel on $S$ directly, and the accumulate-and-decide form of \eqref{eq:box} applies to them unchanged, while Proposition~\ref{prop:kronecker} does not.

\section{Graded modality}\label{sec:graded}

Normalized by the mass of the accessible set, the accumulation of Definition~\ref{def:accumulation} is a probability; thresholds (other than one and zero) give the probability operators of the semantics of graded modality.

\begin{definition}[Probability operators]\label{def:probop}
    Let $(D, \mathcal{M}, \mathcal{R}, \mu)$ be a measure frame with $0 < \mu_w(\mathcal{R}(w)) < \infty$ for every $w$, and let $\pi_w(X) := \mu_w(\mathcal{R}(w) \cap X) / \mu_w(\mathcal{R}(w))$ be the conditional probability of $X$ on the accessible set. For $\theta \in [0, 1]$ the \textbf{probability operators} are
    $$(\Pi_{\ge \theta}\varphi)(w) = 1 \iff \pi_w(\varphi) \ge \theta, \qquad (\Pi_{> \theta}\varphi)(w) = 1 \iff \pi_w(\varphi) > \theta.$$
\end{definition}

In the form of \eqref{eq:box}, $\Pi_{\ge\theta}\varphi$ holds at $w$ if and only if $(\bA\varphi)(w) \ge \theta\, (\bA\ind)(w)$: the same accumulation, with the decision moved from the endpoints of the unit interval to $\theta$.

\begin{proposition}[Probability operators]\label{prop:probop}
    On every measure frame with $0 < \mu_w(\mathcal{R}(w)) < \infty$:
    \begin{enumerate}
    	\item $\Pi_{\ge 1} = \Box$ and $\Pi_{> 0} = \Diamond$;
    	\item $\neg \Pi_{\ge\theta} \neg\varphi = \Pi_{> 1 - \theta}\, \varphi$ and $\neg \Pi_{>\theta} \neg\varphi = \Pi_{\ge 1 - \theta}\, \varphi$; in particular $\neg \Pi_{> 1/2} \neg\varphi = \Pi_{\ge 1/2}\, \varphi$;
    	\item each $\Pi_{\ge\theta}$ and $\Pi_{>\theta}$ is monotone: if the truth set of $\varphi$ is contained in that of $\psi$, then $\Pi\varphi \to \Pi\psi$ is valid;
    	\item $\Pi_{\ge\theta}$ validates $\Pi_{\ge\theta}\varphi \wedge \Pi_{\ge\theta}\psi \to \Pi_{\ge\theta}(\varphi \wedge \psi)$ on every such frame if and only if $\theta \in \{0, 1\}$; for $0 < \theta < 1$ there is a discrete frame on which it fails.
    \end{enumerate}
\end{proposition}

\begin{proof}[Proof sketch]
    \noindent (1) $\pi_w(\varphi) \ge 1$ if and only if $\mu_w(\mathcal{R}(w) \setminus \varphi) = 0$, and $\pi_w(\varphi) > 0$ if and only if $\mu_w(\mathcal{R}(w) \cap \varphi) > 0$. 
    
    \noindent (2) $\pi_w(\neg\varphi) = 1 - \pi_w(\varphi)$, so $\pi_w(\neg\varphi) < \theta$ if and only if $\pi_w(\varphi) > 1 - \theta$, and $\pi_w(\neg\varphi) \le \theta$ if and only if $\pi_w(\varphi) \ge 1 - \theta$. 
    
    \noindent (3) is monotonicity of measures. 
    
    \noindent (4) For $\theta = 1$, the operator is $\Box$, and the axiom is an instance of Theorem~\ref{thm:normal}; for $\theta = 0$ the operator is constantly true. For $0 < \theta < 1$ choose $n \ge 2$ minimal with $1 - 1/n \ge \theta$, so that $1 - 2/n \le 1 - 1/(n-1) < \theta$ when $n \ge 3$ and $1 - 2/n = 0 < \theta$ when $n = 2$; take $n$ worlds each accessible from a world $w$ with counting measure, $\varphi$ true at all but one of them and $\psi$ true at all but another: then $\pi_w(\varphi) = \pi_w(\psi) = 1 - 1/n \ge \theta$ while $\pi_w(\varphi \wedge \psi) = 1 - 2/n < \theta$.
\end{proof}

Item (4) is the familiar non-normality of probability operators \cite{Yalcin2010}: the logic of $\Pi_{\ge\theta}$ is a monotone modal logic, normal only at the endpoints, where the operator is the box or trivial. The almost-everywhere box of Section~\ref{sec:modalops} is the member of the family with threshold one, and the only nontrivial normal member.

In the semantics of natural-language modality, \cite{Yalcin2010} interprets \textit{probably $\varphi$} as truth of $\varphi$ with probability above one half on the modal base, which is $\Pi_{> 1/2}$ with $\pi_w$ the probability measure of the epistemic state at $w$; comparative expressions such as \textit{$\varphi$ is more likely than $\psi$} compare $\pi_w(\varphi)$ with $\pi_w(\psi)$; and \cite{Lassiter2017} develops the graded semantics in which the degrees $\pi_w(\varphi)$ are the values of a scale on which the epistemic modals sit. In Kratzer's semantics \cite{Kratzer1991}, modals quantify over the best worlds of the modal base under an ordering source; the measure frame replaces the ordering by a measure on the base, and the relation between the two is the subject of \cite{Lassiter2017}. Where the strength of \textit{must} is concerned, \cite{vonFintelGillies2010} argue that \textit{must $\varphi$} requires $\varphi$ at every world of the modal base, which, on a discrete base, is the box of Definition~\ref{def:intendenot} under counting measure, while \cite{Lassiter2016} argues for a weak reading on which \textit{must $\varphi$} is $\Pi_{\ge\theta}\varphi$ for a high threshold $\theta < 1$. The almost-everywhere box is $\Pi_{\ge 1}$: strong in Lassiter's scale, and tolerant only of exceptions that form a null set, which on a discrete base are no exceptions at all. Which reading English selects is a disputed empirical question; the present construction supplies the algebra for either.

\section{Conclusion}\label{sec:conclusion}

Typed intensional domains embed injectively into a free carrier and an operator carrier for each type: on the free carrier, every semantic function lifts to a unique multilinear map with composition preserved; on the operator carrier, functions become linear operators, intensions become operators from the index space, function application becomes bilinear evaluation. The two agree on primitive types, and on function types, the operator carrier compresses the free one. The descent theorems of Theorems~\ref{thm:descentsets} and \ref{thm:descentfinite} characterize preserving a Boolean-valued functional of a power set acts linearly on encoded sets if and only if it is constant, an ultrafilter indicator, or the complement of one, which (on finite domains) are Montague's individuals; determiners over restrictors of two or more elements, modal operators over two or more accessible indices or over an atomless measure of positive mass, and attitude operators over two or more alternatives fall outside the linear regime, and compute, instead, by linear accumulation followed by a decision.

Counting measure recovers Kripke semantics at every cardinality, and the accumulate-and-decide form \eqref{eq:box} is the tree-of-numbers form of quantification over indices. Under measures, the logic remains normal, while seriality, reflexivity, and symmetry become conditions on positive mass and atoms, and transitivity and euclideanness keep their axioms; the logic of ``almost surely'' on the universal frame is KD45. Compound indices factor the accessibility kernel as a Kronecker product, which yields the commutation axioms of product logics; the non-product case is, for example, the light cone. With a probability measure on the modal base, the measure-based operators are the probability operators of graded modality.

Determiners, modals, and attitude verbs act on operator encodings by bilinear accumulation and a check for a decision. The measure-theoretic box is the threshold-one member of a family, whose linguistic assignment is disputed. The hyperintensional refinement, in which necessarily equivalent propositions come apart, requires an encoding finer than the operator $\bP_\varphi$. The proof theory of measure frames, in particular, the converse directions of the correspondence proposition (up to null sets) and the axiomatization of the classes of frames in which the atom condition of Proposition~\ref{prop:modaldescent} holds, remains for future work.

\bibliography{intensional}

@incollection{allen1997actions,
  author    = {Allen, James F. and Ferguson, George},
  editor    = {Stock, Oliviero},
  title     = {Actions and Events in Interval Temporal Logic},
  booktitle = {Spatial and Temporal Reasoning},
  year      = {1997},
  publisher = {Springer Netherlands},
  address   = {Dordrecht},
  pages     = {205--245},
  isbn      = {978-0-585-28322-7},
  doi       = {10.1007/978-0-585-28322-7_7}
}

@article{Amigo2022,
    author = {Amigó, Enrique and Ariza-Casabona, Alejandro and Fresno, Victor and Martí, M. Antònia},
    title = {Information Theory–based Compositional Distributional Semantics},
    journal = {Computational Linguistics},
    volume = {48},
    number = {4},
    pages = {907--948},
    year = {2022},
    month = {12},
    issn = {0891-2017},
    doi = {10.1162/coli_a_00454},
}

@incollection{andreka2007logic,
  author    = {Andr{\'e}ka, Hajnal and Madar{\'a}sz, Judit X. and N{\'e}meti, Istv{\'a}n},
  title     = {Logic of Space-Time and Relativity Theory},
  booktitle = {Handbook of Spatial Logics},
  editor    = {Aiello, Marco and Pratt-Hartmann, Ian and Van Benthem, Johan},
  publisher = {Springer Netherlands},
  address   = {Dordrecht},
  pages     = {607--711},
  year      = {2007}
}

@article{Andreoletti2024,
  author  = {Andreoletti, Guido},
  title   = {Branching Time, Fatalism, and Possibilities},
  journal = {KRITERION -- Journal of Philosophy},
  volume  = {38},
  number  = {3--4},
  pages   = {139--155},
  year    = {2024},
  doi     = {10.1515/krt-2024-0024},
}

@inproceedings{Arora2012,
    author = {Arora, Sanjeev and Ge, Rong and Moitra, Ankur},
    title = {Learning Topic Models -- Going beyond SVD},
    year = {2012},
    isbn = {9780769548746},
    publisher = {IEEE Computer Society},
    address = {USA},
    doi = {10.1109/FOCS.2012.49},
    booktitle = {Proceedings of the 2012 IEEE 53rd Annual Symposium on Foundations of Computer Science},
    pages = {1–10},
    numpages = {10},
    series = {FOCS '12}
}

@article{Baroni2014,
  author    = {Baroni, Marco and Bernardi, Raffaella and Zamparelli, Roberto},
  title     = {Frege in Space: A Program for Composition Distributional Semantics},
  journal   = {Linguistic Issues in Language Technology},
  volume    = {9},
  year      = {2014},
  publisher = {CSLI Publications}
}

@book{Bishop2006,
  author    = {Bishop, Christopher M.},
  title     = {Pattern Recognition and Machine Learning},
  publisher = {Springer},
  address   = {New York},
  year      = {2006}
}

@article{Boleda2020,
  title={Distributional Semantics and Linguistic Theory},
   volume={6},
   ISSN={2333-9691},
   DOI={10.1146/annurev-linguistics-011619-030303},
   number={1},
   journal={Annual Review of Linguistics},
   publisher={Annual Reviews},
   author={Boleda, Gemma},
   year={2020},
   month=Jan, 
   pages={213–234} 
}

@article{Bullinaria2012,
  author  = {Bullinaria, John A. and Levy, Joseph P.},
  title   = {Extracting semantic representations from word co-occurrence statistics: stop-lists, stemming, and {SVD}},
  journal = {Behavior Research Methods},
  volume  = {44},
  pages   = {890--907},
  year    = {2012},
  doi     = {10.3758/s13428-011-0183-8},
}

@incollection{Burnistov2023,
  author    = {Burnistov, Artem and Stukachev, Alexey},
  editor    = {Loukanova, Roussanka and Lumsdaine, Peter LeFanu and Muskens, Reinhard},
  title     = {Generalized Computable Models and Montague Semantics},
  booktitle = {Logic and Algorithms in Computational Linguistics 2021 (LACompLing2021)},
  series    = {Studies in Computational Intelligence},
  year      = {2023},
  publisher = {Springer},
  address   = {Cham},
  pages     = {107--124},
  isbn      = {9783031217791},
  doi       = {10.1007/978-3-031-21780-7_5}
}

@incollection{cardelli2002spatial,
  author    = {Cardelli, Luca and Gardner, Philippa and Ghelli, Giorgio},
  editor    = {Widmayer, Peter and Eidenbenz, Stephan and Triguero, Francisco and Morales, Rafael and Conejo, Ricardo and Hennessy, Matthew},
  title     = {A Spatial Logic for Querying Graphs},
  booktitle = {Automata, Languages and Programming},
  series    = {Lecture Notes in Computer Science},
  volume    = {2380},
  year      = {2002},
  publisher = {Springer},
  address   = {Berlin, Heidelberg},
  doi       = {10.1007/3-540-45465-9_51},
}

@inproceedings{cardelli2000anytime,
    author = {Cardelli, Luca and Gordon, Andrew D.},
    title = {Anytime, anywhere: modal logics for mobile ambients},
    year = {2000},
    isbn = {1581131259},
    publisher = {Association for Computing Machinery},
    address = {New York, NY, USA},
    doi = {10.1145/325694.325742},
    booktitle = {Proceedings of the 27th ACM SIGPLAN-SIGACT Symposium on Principles of Programming Languages},
    pages = {365--377},
    numpages = {13},
    location = {Boston, MA, USA},
    series = {POPL '00}
}

@book{Chang1977,
  author    = {Chang, C. C. and Keisler, H. J.},
  title     = {Model Theory},
  series    = {Dover Books on Mathematics},
  year      = {2012},
  publisher = {Dover Publications},
  isbn      = {9780486488219},
  lccn      = {2011041607},
}

@book{Chellas1980,
  author    = {Chellas, Brian F.},
  title     = {Modal Logic: An Introduction},
  publisher = {Cambridge University Press},
  address   = {Cambridge},
  year      = {1980}
}

@book{Chierchia2000,
  author    = {Chierchia, Gennaro and McConnell-Ginet, Sally},
  title     = {Meaning and Grammar: An Introduction to Semantics},
  publisher = {MIT Press},
  address   = {Cambridge, MA},
  year      = {2000}
}

@article{Church1940,
 ISSN = {00224812},
 author = {Alonzo Church},
 journal = {The Journal of Symbolic Logic},
 number = {2},
 pages = {56--68},
 publisher = {Association for Symbolic Logic},
 title = {A Formulation of the Simple Theory of Types},
 volume = {5},
 year = {1940}
}

@article{Coecke2013,
    title = {Lambek vs. Lambek: Functorial vector space semantics and string diagrams for Lambek calculus},
    journal = {Annals of Pure and Applied Logic},
    volume = {164},
    number = {11},
    pages = {1079-1100},
    year = {2013},
    note = {Special issue on Seventh Workshop on Games for Logic and Programming Languages (GaLoP VII)},
    issn = {0168-0072},
    doi = {10.1016/j.apal.2013.05.009},
    author = {Bob Coecke and Edward Grefenstette and Mehrnoosh Sadrzadeh},
}

@article{Coecke2010,
  author    = {Coecke, Bob and Sadrzadeh, Mehrnoosh and Clark, Stephen},
  title     = {Mathematical Foundations for a Compositional Distributional Model of Meaning},
  journal   = {Linguistic Analysis},
  note      = {Lambek Festschrift, Special Issue},
  year      = {2010},
  eprint    = {1003.4394},
  archiveprefix = {arXiv}
}

@book{Conway2007,
  author    = {Conway, John B.},
  title     = {A Course in Functional Analysis},
  series    = {Graduate Texts in Mathematics},
  volume    = {96},
  publisher = {Springer},
  address   = {New York},
  year      = {2007},
  doi       = {10.1007/978-1-4757-4383-8}
}

@misc{Coppock2024,
  author       = {Coppock, Elizabeth and Champollion, Lucas},
  title        = {Invitation to Formal Semantics},
  year         = {2024},
  url          = {https://eecoppock.info/bootcamp/semantics-boot-camp.pdf},
  note         = {[Online; accessed 25-May-2024]}
}

@article{firth2957,
    author = {Firth, John Rupert},
    title = {Applications of General Linguistics},
    journal = {Transactions of the Philological Society},
    volume = {56},
    number = {1},
    pages = {1--14},
    doi = {10.1111/j.1467-968X.1957.tb00568.x},
    year = {1957}
}

@incollection{gabbay2003many,
    title = {Chapter 3 - Many-dimensional modal logics},
    editor = {D.M. Gabbay and A. Kurucz and F. Wolter and M. Zakharyaschev},
    series = {Studies in Logic and the Foundations of Mathematics},
    author ={D.M. Gabbay and A. Kurucz and F. Wolter and M. Zakharyaschev},
    publisher = {Elsevier},
    volume = {148},
    pages = {111-193},
    year = {2003},
    booktitle = {Many-Dimensional Modal Logics},
    issn = {0049-237X},
    doi = {10.1016/S0049-237X(03)80004-6},
}

@book{Gallin2016,
  title={Intensional and Higher-Order Modal Logic},
  author={Gallin, D.},
  isbn={9781483274737},
  year={2016},
  publisher={North Holland}
}

@book{Gardenfors2014,
  author    = {G{\"a}rdefors, Peter},
  title     = {The Geometry of Meaning: Semantics Based on Conceptual Spaces},
  publisher = {The MIT Press},
  year      = {2014}
}

@book{Gardenfors2000,
  author    = {G{\"a}rdefors, Peter},
  title     = {Conceptual Spaces: The Geometry of Thought},
  publisher = {The MIT Press},
  year      = {2000}
}

@article{Greco2020,
  author       = {Giuseppe Greco and
                  Fei Liang and
                  Michael Moortgat and
                  Alessandra Palmigiano and
                  Apostolos Tzimoulis},
  title        = {Vector Spaces as Kripke Frames},
  journal      = {{FLAP}},
  volume       = {7},
  number       = {5},
  pages        = {853--873},
  year         = {2020},
}

@inproceedings{grefenstette2013,
    title = {Towards a Formal Distributional Semantics: Simulating Logical Calculi with Tensors},
    author = {Grefenstette, Edward},
    editor = {Diab, Mona  and Baldwin, Tim  and Baroni, Marco},
    booktitle = {Second Joint Conference on Lexical and Computational Semantics (*{SEM}), Volume 1: Proceedings of the Main Conference and the Shared Task: Semantic Textual Similarity},
    month = jun,
    year = {2013},
    publisher = {Association for Computational Linguistics},
    pages = {1--10}
}

@book{harris1979mathematical,
  title={Mathematical Structures of Language},
  author={Harris, Zellig Sabbettai},
  isbn={9780882759586},
  lccn={79004568},
  series={Interscience tracts in pure and applied mathematics},
  year={1979},
  publisher={R. E. Krieger Publishing Company}
}

@article{HedgesSadrzadeh2019,
  author  = {Hedges, Jules and Sadrzadeh, Mehrnoosh},
  title   = {A Generalised Quantifier Theory of Natural Language in Categorical Compositional Distributional Semantics with Bialgebras},
  journal = {Mathematical Structures in Computer Science},
  volume  = {29},
  number  = {6},
  pages   = {783--809},
  year    = {2019},
  doi     = {10.1017/S0960129518000324}
}

@book{Heim1998,
  author    = {Heim, Irene and Kratzer, Angelika},
  title     = {Semantics in Generative Grammar},
  publisher = {Blackwell},
  year      = {1998}
}

@article{Hirsch2018,
     ISSN = {00224812, 19435886},
     author = {Hirsch, Robin and Reynolds, Mark},
     journal = {The Journal of Symbolic Logic},
     number = {3},
     pages = {829--867},
     publisher = {[Association for Symbolic Logic, Cambridge University Press]},
     title = {The Temporal Logic of Two Dimensional Minkowski Spacetime is Decidable},
     volume = {83},
     year = {2018}
}

@article{HIRSHFELD20121,
    title = {Continuous time temporal logic with counting},
    journal = {Information and Computation},
    volume = {214},
    pages = {1--9},
    year = {2012},
    issn = {0890-5401},
    doi = {10.1016/j.ic.2011.11.003},
    author = {Yoram Hirshfeld and Alexander Rabinovich},
}

@article{Kamide2005,
    author = {Kamide, Norihiro},
    title = {A spatial modal logic with a location interpretation},
    journal = {Mathematical Logic Quarterly},
    volume = {51},
    number = {4},
    pages = {331--341},
    doi = {10.1002/malq.200510001},
    year = {2005}
}

@article{KeenanStavi1986,
  author  = {Keenan, Edward L. and Stavi, Jonathan},
  title   = {A Semantic Characterization of Natural Language Determiners},
  journal = {Linguistics and Philosophy},
  volume  = {9},
  number  = {3},
  pages   = {253--326},
  year    = {1986},
  doi     = {10.1007/BF00630273}
}

@incollection{kontchakov2007spatial,
  author    = {Kontchakov, Roman and Kurucz, Agi and Wolter, Frank and Zakharyaschev, Michael},
  editor    = {Aiello, Marco and Pratt-Hartmann, Ian and van Benthem, Johan},
  title     = {Spatial Logic + Temporal Logic = ?},
  booktitle = {Handbook of Spatial Logics},
  year      = {2007},
  publisher = {Springer},
  address   = {Dordrecht},
  doi       = {10.1007/978-1-4020-5587-4_9},
}

@incollection{Kratzer1991,
  author    = {Kratzer, Angelika},
  title     = {Modality},
  booktitle = {Semantics: An International Handbook of Contemporary Research},
  editor    = {von Stechow, Arnim and Wunderlich, Dieter},
  publisher = {de Gruyter},
  address   = {Berlin},
  pages     = {639--650},
  year      = {1991}
}

@incollection{KURUCZ2007869,
    title = {Combining modal logics},
    editor = {Patrick Blackburn and Johan {Van Benthem} and Frank Wolter},
    series = {Studies in Logic and Practical Reasoning},
    publisher = {Elsevier},
    volume = {3},
    pages = {869--924},
    year = {2007},
    booktitle = {Handbook of Modal Logic},
    issn = {1570-2464},
    doi = {10.1016/S1570-2464(07)80018-8},
    author = {Agi Kurucz},
}

@article{Lambek1958,
    author = {Joachim Lambek},
    title = {The Mathematics of Sentence Structure},
    journal = {The American Mathematical Monthly},
    volume = {65},
    number = {3},
    pages = {154--170},
    year = {1958},
    publisher = {Taylor \& Francis},
    doi = {10.1080/00029890.1958.11989160},
}

@incollection{Lambek1961,
  author    = {Lambek, Joachim},
  title     = {On the Calculus of Syntactic Types},
  booktitle = {Structure of Language and its Mathematical Aspects},
  editor    = {Jakobson, Roman},
  publisher = {American Mathematical Society},
  address   = {Providence, RI},
  year      = {1961},
  pages     = {166--178}
}

@article{Lando2012,
  author  = {Lando, Tamar},
  title   = {Completeness of {S4} for the {L}ebesgue Measure Algebra},
  journal = {Journal of Philosophical Logic},
  volume  = {41},
  number  = {2},
  pages   = {287--316},
  year    = {2012},
  doi     = {10.1007/s10992-010-9161-3}
}

@book{Lassiter2017,
  author    = {Lassiter, Daniel},
  title     = {Graded Modality: Qualitative and Quantitative Perspectives},
  publisher = {Oxford University Press},
  address   = {Oxford},
  year      = {2017}
}

@article{Lassiter2016,
  author  = {Lassiter, Daniel},
  title   = {Must, Knowledge, and (In)directness},
  journal = {Natural Language Semantics},
  volume  = {24},
  number  = {2},
  pages   = {117--163},
  year    = {2016},
  doi     = {10.1007/s11050-016-9121-8}
}

@article{Lenci2022,
  author    = {Lenci, Alessandro and Sahlgren, Magnus and Jeuniaux, Patrick and Gyllensten, Amaru Cuba and Miliani, Martina},
  title     = {A Comparative Evaluation and Analysis of Three Generations of Distributional Semantic Models},
  journal   = {Language Resources and Evaluation},
  volume    = {56},
  number    = {4},
  pages     = {1269--1313},
  year      = {2022},
  doi       = {10.1007/s10579-021-09575-z}
}

@book{Lenci2023,
  title={Distributional semantics},
  author={Lenci, Alessandro and Sahlgren, Magnus},
  year={2023},
  publisher={Cambridge University Press},
  doi   = {10.1017/9780511783692},
  isbn  = {9780511783692},
}

@misc{Lewis2026,
      title={Some Results on Causal Modalities in General Spacetimes}, 
      author={Marco Lewis and Nesta van der Schaaf},
      year={2026},
      eprint={2601.14029},
      archivePrefix={arXiv},
      primaryClass={math.LO},
      url={https://arxiv.org/abs/2601.14029}, 
}

@inproceedings{Mao2022,
  author    = {Mao, Shufan and Huebner, Philip A. and Willits, Jon A.},
  title     = {Compositional Generalization in a Graph-based Model of Distributional Semantics},
  booktitle = {Proceedings of the 44th Annual Meeting of the Cognitive Science Society},
  year      = {2022}
}

@incollection{Marcolli2023,
    author = {Marcolli, Matilde and Chomsky, Noam and Berwick, Robert C.},
    isbn = {9780262383349},
    title = {The Syntax–Semantics Interface: An Algebraic Model},
    booktitle = {Mathematical Structure of Syntactic Merge: An Algebraic Model for Generative Linguistics},
    publisher = {The MIT Press},
    year = {2025},
    month = {08},
    doi = {10.7551/mitpress/15811.003.0008},
}

@book{Mendelson1997,
  author    = {Mendelson, Elliott},
  title     = {Introduction to Mathematical Logic},
  edition   = {4},
  series    = {Discrete Mathematics and Its Applications},
  publisher = {Taylor \& Francis},
  year      = {1997}
}

@misc{mikolov2013efficientestimationwordrepresentations,
      title={Efficient Estimation of Word Representations in Vector Space}, 
      author={Tomas Mikolov and Kai Chen and Greg Corrado and Jeffrey Dean},
      year={2013},
      eprint={1301.3781},
      archivePrefix={arXiv},
      primaryClass={cs.CL},
      url={https://arxiv.org/abs/1301.3781}, 
}

@inproceedings{MitchellLapata2008,
    title = {Vector-based Models of Semantic Composition},
    author = {Mitchell, Jeff  and
      Lapata, Mirella},
    editor = {Moore, Johanna D.  and
      Teufel, Simone  and
      Allan, James  and
      Furui, Sadaoki},
    booktitle = {Proceedings of ACL-08: HLT},
    month = jun,
    year = {2008},
    address = {Columbus, Ohio},
    publisher = {Association for Computational Linguistics},
    pages = {236--244}
}

@article{Mizraji1992,
  author  = {Mizraji, Eduardo},
  title   = {Vector Logics: The Matrix-Vector Representation of Logical Calculus},
  journal = {Fuzzy Sets and Systems},
  volume  = {50},
  number  = {2},
  pages   = {179--185},
  year    = {1992}
}

@article{Mizraji1996,
  author  = {Mizraji, Eduardo},
  title   = {The Operators of Vector Logic},
  journal = {Mathematical Logic Quarterly},
  volume  = {42},
  pages   = {27--40},
  year    = {1996}
}

@article{Mizraji2008,
  author  = {Mizraji, Eduardo},
  title   = {Vector Logic: A Natural Algebraic Representation of the Fundamental Logical Gates},
  journal = {Journal of Logic and Computation},
  volume  = {18},
  pages   = {97--121},
  year    = {2008}
}

@incollection{Montague1974,
  author    = {Montague, Richard},
  title     = {English as a Formal Language},
  booktitle = {Formal Philosophy: Selected Papers of Richard Montague},
  editor    = {Thomason, Richmond},
  pages     = {188--221},
  publisher = {Yale University Press},
  address   = {New Haven, CT},
  year      = {1974}
}

@article{morales2007new,
  author  = {Morales, Antonio and Navarrete, Isabel and Sciavicco, Guido},
  title   = {A New Modal Logic for Reasoning about Space: Spatial Propositional Neighborhood Logic},
  journal = {Annals of Mathematics and Artificial Intelligence},
  volume  = {51},
  number  = {1},
  pages   = {1--25},
  year    = {2007},
  doi     = {10.1007/s10472-007-9083-0}
}

@Inbook{Pacuit2017,
    author={Pacuit, Eric},
    title={Core Theory},
    bookTitle={Neighborhood Semantics for Modal Logic},
    year={2017},
    publisher={Springer International Publishing},
    address={Cham},
    pages={39--96},
    isbn={978-3-319-67149-9},
    doi={10.1007/978-3-319-67149-9_2},
}

@InProceedings{Palmigiano2024,
    author={Palmigiano, Alessandra and Panettiere, Mattia and Switrayni, Ni Wayan},
    editor={Metcalfe, George and Studer, Thomas and de Queiroz, Ruy},
    title={Correspondence Theory on Vector Spaces},
    booktitle={Logic, Language, Information, and Computation},
    year={2024},
    publisher={Springer Nature Switzerland},
    address={Cham},
    pages={140--156},
    isbn={978-3-031-62687-6}
}

@article{Pavlick2022,
   author = {Pavlick, Ellie},
   title = {Semantic Structure in Deep Learning}, 
   journal= {Annual Review of Linguistics},
   year = {2022},
   volume = {8},
   pages = {447--471},
   doi = {10.1146/annurev-linguistics-031120-122924},
   publisher = {Annual Reviews},
   issn = {2333-9691},
  }

@inproceedings{pennington2014glove,
  author = {Jeffrey Pennington and Richard Socher and Christopher D. Manning},
  booktitle = {Empirical Methods in Natural Language Processing (EMNLP)},
  title = {GloVe: Global Vectors for Word Representation},
  year = {2014},
  pages = {1532--1543},
  url = {http://www.aclweb.org/anthology/D14-1162},
}

@article{ploug2012branching,
  author  = {Ploug, Thomas and {\O}hrstr{\o}m, Peter},
  title   = {Branching Time, Indeterminism, and Tense Logic},
  journal = {Synthese},
  volume  = {188},
  number  = {3},
  pages   = {367--379},
  year    = {2012},
  doi       ={10.1007/s11229-011-9944-2}
}

@book{prior1957time,
  author    = {Prior, Arthur N.},
  title     = {Time and Modality},
  publisher = {Oxford University Press},
  address   = {Oxford},
  year      = {1957}
}

@article{Quigley2025,
  author    = {Quigley, Daniel},
  title     = {A Vector Logic for Extensional Formal Semantics},
  journal   = {Journal of Logic, Language and Information},
  year      = {2025},
  volume    = {34},
  number    = {5},
  pages     = {557--599},
  doi       = {10.1007/s10849-025-09443-x},
  url       = {10.1007/s10849-025-09443-x},
  issn      = {1572-9583},
}

@Inbook{Rescher1968,
 author = {Nicholas Rescher},
 journal = {The Journal of Symbolic Logic},
 pages = {229--249},
 publisher = {Springer Netherlands},
 title = {Topological Logic},
 bookTitle={Topics in Philosophical Logic},
 volume = {33},
 year = {1968},
 isbn={978-94-017-3546-9},
doi={10.1007/978-94-017-3546-9_13},
ddress={Dordrecht}
}

@ARTICLE{reynolds2001products,
  author={Reynolds, Mark and Zakharyaschev, Michael},
  journal={Journal of Logic and Computation}, 
  title={On the Products of Linear Modal Logics}, 
  year={2001},
  volume={11},
  number={6},
  pages={909--931},
  doi={10.1093/logcom/11.6.909}}

@article{Simons2006,
	author = {Peter Simons},
	doi = {10.1007/s11229-005-5517-6},
	journal = {Synthese},
	number = {3},
	pages = {443--458},
	title = {The Logic of Location},
	volume = {150},
	year = {2006}
}

@article{Smolensky1990,
title = {Tensor product variable binding and the representation of symbolic structures in connectionist systems},
journal = {Artificial Intelligence},
volume = {46},
number = {1},
pages = {159-216},
year = {1990},
issn = {0004-3702},
doi = {10.1016/0004-3702(90)90007-M},
author = {Paul Smolensky},
}

@incollection{Tarski1931,
  author    = {Tarski, Alfred},
  title     = {The Concept of Truth in Formalized Languages},
  booktitle = {Logic, Semantics, Metamathematics},
  editor    = {Tarski, Alfred},
  pages     = {152--278},
  publisher = {Oxford University Press},
  year      = {1931}
}

@book{vanBenthem1986,
  author    = {van Benthem, Johan},
  title     = {Essays in Logical Semantics},
  publisher = {Springer Dordrecht},
  address   = {Dordrecht},
  series    = {Studies in Linguistics and Philosophy},
  volume    = {29},
  year      = {1986},
  doi       = {10.1007/978-94-009-4540-1}
}

@incollection{vanbenthem2007modal,
  author    = {van Benthem, Johan and Bezhanishvili, Guram},
  editor    = {Aiello, Marco and Pratt-Hartmann, Ian and van Benthem, Johan},
  title     = {Modal Logics of Space},
  booktitle = {Handbook of Spatial Logics},
  year      = {2007},
  publisher = {Springer},
  address   = {Dordrecht},
  doi       = {10.1007/978-1-4020-5587-4_5},
}

@incollection{Venema2008,
	author = {Yde Venema},
	booktitle = {The Blackwell Guide to Philosophical Logic},
	editor = {Lou Goble},
	pages = {203--223},
	publisher = {Wiley-Blackwell},
	title = {Temporal Logic},
	year = {2008}
}

@article{vonFintelGillies2010,
  author  = {von Fintel, Kai and Gillies, Anthony S.},
  title   = {Must \ldots Stay \ldots Strong!},
  journal = {Natural Language Semantics},
  volume  = {18},
  number  = {4},
  pages   = {351--383},
  year    = {2010},
  doi     = {10.1007/s11050-010-9058-2}
}

@InProceedings{walega2019,
  author =	{Wa{\l}\k{e}ga, Przemys{\l}aw Andrzej and Zawidzki, Micha{\l}},
  title =	{{A Modal Logic for Subject-Oriented Spatial Reasoning}},
  booktitle =	{26th International Symposium on Temporal Representation and Reasoning (TIME 2019)},
  pages =	{4:1--4:22},
  series =	{Leibniz International Proceedings in Informatics (LIPIcs)},
  ISBN =	{978-3-95977-127-6},
  ISSN =	{1868-8969},
  year =	{2019},
  volume =	{147},
  editor =	{Gamper, Johann and Pinchinat, Sophie and Sciavicco, Guido},
  publisher =	{Schloss Dagstuhl -- Leibniz-Zentrum f{\"u}r Informatik},
  address =	{Dagstuhl, Germany},
  doi =		{10.4230/LIPIcs.TIME.2019.4},
}

@article{Westphal2005,
    author = {Westphal, Jonathan and Hardy, Jim},
    title = {Logic as a Vector System},
    journal = {Journal of Logic and Computation},
    volume = {15},
    number = {5},
    pages = {751--765},
    year = {2005},
    month = {10},
    issn = {0955-792X},
    doi = {10.1093/logcom/exi040},
}

@article{Yalcin2010,
  author  = {Yalcin, Seth},
  title   = {Probability Operators},
  journal = {Philosophy Compass},
  volume  = {5},
  number  = {11},
  pages   = {916--937},
  year    = {2010},
  doi     = {10.1111/j.1747-9991.2010.00360.x}
}

@unpublished{vonFintel2023,
  author    = {{von Fintel}, Kai and Heim, Irene},
  title     = {Intensional Semantics},
  note      = {MIT lecture notes},
  year      = {2023},
  url       = {https://github.com/fintelkai/fintel-heim-intensional-notes}
}

\end{document}